\documentclass{aimsbis}

\usepackage[all]{xy} 

\usepackage{pstricks,pst-plot}
\usepackage{amsmath}
\usepackage{amsmath,amsfonts,amsthm,epsfig,latexsym, amssymb}

\usepackage{paralist}
\usepackage{graphics} 
\usepackage{epsfig} 
\usepackage{color,subfigure}
 \usepackage[colorlinks=true,backref]{hyperref}
\hypersetup{urlcolor=blue, citecolor=red}
\usepackage{hyperref}
\usepackage[capitalise]{cleveref} 

\theoremstyle{plain}
    \newtheorem{theorem}{Theorem}
    
     \newtheorem{theoremnum}{Theorem}[section]
    \newtheorem{corollary}[theoremnum]{Corollary}
    \newtheorem{proposition}[theoremnum]{Proposition}
    \newtheorem{lemma}[theoremnum]{Lemma}
    \newtheorem{conjecture}[theoremnum]{Conjecture}

    \newtheorem{scholium}[theoremnum]{Scholium}
    \newtheorem{question}[theoremnum]{Question}

\newtheorem{claim}{Claim}

\theoremstyle{definition}
    \newtheorem{definition}[theoremnum]{Definition}
       \newtheorem{example}[theoremnum]{Example}

\theoremstyle{remark}
    \newtheorem{remark}[theoremnum]{Remark}

\def\AA{{\mathbb A}}  \def\CC{{\mathbb C}} \def\DD{{\mathbb D}}
\def\EE{{\mathbb E}} \def\FF{{\mathbb F}} \def\GG{{\mathbb G}} \def\HH{{\mathbb H}}
   
 \def\NN{{\mathbb N}}  \def\PP{{\mathbb P}}
\def\QQ{{\mathbb Q}} \def\RR{{\mathbb R}} \def\SS{{\mathbb S}} \def\TT{{\mathbb T}}
   
 \def\ZZ{{\mathbb Z}}

\newcommand{\tA}{\mathtt{A}}\newcommand{\tB}{\mathtt{B}}

\def\cA{{\mathcal A}}   \def\cM{{\mathcal M}} \def\cS{{\mathcal S}}
\def\cB{{\mathcal B}}    
\def\cC{{\mathcal C}}  \def\cI{{\mathcal I}} \def\cO{{\mathcal O}} \def\cU{{\mathcal U}}
  \def\cJ{{\mathcal J}}  
\def\cE{{\mathcal E}}    
  \def\cL{{\mathcal L}} \def\cR{{\mathcal R}}

\def\supp{\operatorname{supp}}

\def\GL{\operatorname{GL}}
\def\SL{\operatorname{SL}}
\def\SU{\operatorname{SU}}
\def\SO{\operatorname{SO}}

\def\tr{\operatorname{tr}}

\def\Homeo{\operatorname{Homeo}}
\def\dist{\operatorname{dist}}

\def\Orb{\operatorname{Orb}}

\def\Id{\operatorname{Id}}

\title{Rotation numbers of perturbations of smooth dynamics}

\author[Gourmelon]{}


 \email{}

\thanks{We would like to thank Matheus for bringing to our knowledge the problem from which this paper originated. We would also like to thank Jairo Bochi for interesting conversations. Finally, we would like to thank the Universit\'e d'Orsay, LMO, where part of this paper was written.}

\begin{document}
\maketitle

\centerline{\scshape }
\medskip
{\footnotesize
 \centerline{Nicolas Gourmelon}

 \centerline{Universit\'e de Bordeaux}
 \centerline{CNRS, Bordeaux INP, IMB, UMR 5251, F-33400 Talence, France}
 } 

\bigskip


\renewcommand{\abstractname}{Abstract}

\begin{abstract}
We show how the small perturbations of a linear cocycle have a relative rotation number associated with an invariant measure of the base dynamics an with a $2$-dimensional bundle of the finest dominated splitting (provided that some orientation is preserved). 
Likewise small perturbations of diffeomorphisms have a relative rotation number associated with an invariant measure supported in a  hyperbolic set and with a $2$-dimensional bundle as above.

The properties of that relative rotation number allow some steps towards dichotomies between complex eigenvalues and dominated splittings in higher dimensions and higher regularity. 
We also prove that generic smooth linear cocycles above a full-shift (and actually above infinite factors of transitive subshift of finite type) admit a periodic point with simple Lyapunov spectrum.
\end{abstract}

\maketitle

\section{Introduction}
The idea of using rotation numbers in the study of higher dimensional linear cocycles is not new: Bonatti and Viana~\cite{BV} suggested using them to prove that generic smooth $\SL(n,\RR)$-cocycles over a full shift admit a periodic point with simple Lyapunov spectrum.\footnote{\label{f.simplespectrumdiffeos} This is folklore for diffeomorphisms as an application of the Sternberg linearization lemma: precisely, if  $K$ is an infinite basic hyperbolic set of a smooth diffeomorphisms $f$, then there are arbitrarily small smooth perturbations of $f$ that have a periodic point with simple Lyapunov spectrum in the continuation of $K$. } Carlos Matheus however pointed to us the difficulty of making that argument rigorous. 

The aim of the present paper is to introduce a new notion to fill that gap. It turns out that that rotation number finds many other applications, some of them towards a smooth generic dichotomy between complex eigenvalues and dominated splittings. 

We now give a rough idea of that rotation number. Consider a hyperbolic set $X$ of a diffeomorphism $f$, or alternatively, a linear cocycle $\cA$ fibering on a homeomorphism $T\colon X \to X$.
Consider an invariant measure $\mu$ supported in $X$. Assume that a dominated splitting over the support of $\mu$ admits a $2$-dimensional bundle $E$ along which $Df$, resp. $\cA$, preserve some continuous orientation. 

Given a small $C^1$-path $f_t$ starting at $f$, resp.  $C^0$-path of cocycles $\cA_t$ starting at $\cA$ (all fibering above $T$), the bundle $E$ has a corresponding continuation $E_t$. We see in this paper that the path has a well-defined {\em rotation number above $\mu$ along $E_t$}. That rotation number being invariant under homotopies with fixed extremities, and the spaces of systems we consider being locally simply connected, one can define the {\em relative rotation number} of a perturbation of $f$, resp. $\cA$, {\em along $E$}. From the properties of that rotation number, we deduce properties of generic smooth dynamics:
\begin{itemize}
\item the $C^r$-generic $\SL(n,\RR)$-cocycles over a full shift has a periodic point with simple Lyapunov spectrum, that is, the eigenvalues of the first return linear map have pairwise distinct moduli. This actually holds more generally for cocycles above factors of transitive subshifts of finite type (provided $X$ is infinite).
\item under an assumption of abundance of periodic points in $X$, we make some steps towards a $C^r$-generic dichotomy between dominated splittings and the existence of periodic points with complex eigenvalues. 
\item in the general setting, one may consider the structural $C^r$-stability of the projective cocycle. We prove in some cases that structural $C^r$-stability imply the existence of dominated splittings of all indices.
\end{itemize}

\subsection{Herman's rotation number}
Let $X$ be a compact metric space, let $T\colon X\to X$ be a continuous (not necessarily homeomorphic) transformation, and let $\mathbb{E}=\sqcup_{x\in X} \EE_x$ be a continuous oriented circle bundle with base $X$. 
We consider consider fibered continuous transformations $f\colon \mathbb{E} \to \mathbb{E}$ that project on $T$ and that send the fiber $\EE_x$ on $\EE_{Tx}$ by an orientation-preserving homeomorphism. By a slight abuse of terminology, we say that $f$ is a {\em skew-products (of positive circle homeomorphisms) over base dynamics $T$}. We denote the set of such $f$ by $Sk^+_{\EE,T}$ and endow it with the topology of uniform convergence.

Suppose that $\mathbb{E}=X\times \SS^1$, a skew-product $f$ is homotopic to $(T,\Id)$ within $Sk^+_{\EE,T}$, and $\mu$ is an ergodic measure of $T$ supported in a connected component of $X$. Then there is a well-defined {\em fibered rotation number} of $f$ above $\mu$. It can be seen as an element of the circle $\SS^1$ (see Herman~\cite{H, JS}). The rotation number is a cohomology invariant.

 Following classical terminology, the skew-product $f$ is called {\em mode-locked} above $\mu$ if any skew-product $g$ with same base dynamics $T$ and $C^0$-close enough to $f$ has the same rotation number.  Answering a question of Herman, when $T$ is an irrational rotation of the circle, J\"ager and Bjerkl\"ov~\cite{BJ} characterized mode-locking by the existence of a {\em trapping annulus}, that is, an annulus $S\subset \EE$ whose boundary circles project on $X$ by a finite cover, such that $f^n(S)\Subset S$ for some iterate $n\in \NN$, and such that $f^k(S)\cap S=\emptyset$, for all $0<k<n$.

\subsection{Rotation numbers of paths} 
The objects presented in this section are formally defined in \cref{s.rotmorphism.modelock,s.varyingbundles}.

\subsubsection{For skew-products of circle homeomorphisms}\label{s.rotnumber.isotopies}
If the compact base $X$ is not connected, or if the skew-product $f$ is not isotopic to $\Id$, then $f$ has no intrinsically defined rotation number. However, given a path $\phi=f_t$ of skew-products in $Sk_{\EE,T}^+$ one may consider the path of $k$-th iterates and take its integrated winding number  above a $T$-invariant Borel finite measure $\mu$. This has an asymptotic arithmetic growth rate when $k\to +\infty$, which we call the {\em rotation number} $\rho(\mu,\phi)$ of the path $\phi$ above measure $\mu$. 

It depends only on the isotopy class with fixed extremities of $\phi$, it is linear in $\mu$ and additive for concatenation of paths, and we prove that it is a cohomology invariant, that is, the rotation number only depends on the path of cohomology classes $[f_t]$, where two skew-products are (continuously) {\em cohomologous} if they are conjugate through $h\in Sk^+_{\EE,\Id_X}$.\footnote{If $\mathbb{E}=X\times \SS^1$ then a skew-product $f$ identifies to a pair $(T,{\bf f})$ where 
${\bf f}\colon X\to G=\Homeo^+(\SS^1)$ is a continuous map. It is usually called a {\em $G$-cocycle above $T$}. 
A $G$-cocycle $g$ is then cohomologous to $f$ if there exists a continuous map ${\bf h}\colon X\to G$ such that ${\bf g}(x)={\bf h}(x)\circ {\bf f}(x)\circ {\bf h}^{-1}(x)$ for all $x\in X$.}

That  rotation number extends through projectivization to paths of orientation preserving linear cocycles on 2-dimensional bundles.
In the particular case of $\SL(2,\RR)$-cocycles, Avila and Krikorian~\cite{AK} define a "variation of the fibered rotation number" in a geometrical way. We see in \cref{a.avilakrikorian} that it coincides with ours up to multiplication by $-2$.  

Note that $Sk_{\EE,T}^+$ is locally simply connected, thus a small $C^0$-perturbation $g$ of $f$ in $Sk_{\EE,T}^+$ has a well-defined {\em rotation number 
relative to $f$} obtained by taking a small path $\phi$ joining $f$ to $g$. 
Extending a previous definition, we say that $f$ is {\em mode-locked over $\mu$} if small perturbations have zero relative rotation numbers, and we define a {\em trapping strip} to be a compact subset $S\subset \EE$ such that, for some $n\in \NN$,
\begin{itemize}
\item for all $x$, the set $S\cap \EE_x$ is a disjoint union of $n$ intervals, and each of them contains one of the $n$ intervals forming $f(S)\cap \EE_x$.
\item the map $x\mapsto S\cap \EE_x$ is continuous for the Hausdorff topology.
\end{itemize}
In particular, if the base $X$ is a circle, a trapping strip is a disjoint union of trapping annuli. In \cite{G2}, we generalize the characterization mode-locking of~\cite{BJ} to any minimal dynamics on the base:

\begin{theoremnum}[see \cite{G2}]\label{c.mode.locking} Let $f\in Sk_{\EE,T}^+$ be a skew-product  of positive circle homeomorphisms over a minimal base dynamics $T\colon X\to X$, and let $\mu$ be a $T$-invariant measure. Then the following are equivalent:
\begin{itemize}
\item[$(i)$] $f$ is mode-locked over $\mu$,
\item[$(ii)$] $f$ has a trapping strip.  \end{itemize}
\end{theoremnum}

\begin{remark}\label{r.trappingstri}
It is not clear how much the minimality assumption can be weakened. In the particular case where $f$ is the projectivization of a linear cocycle, we can relax it to asking that $T\colon X\to X$ has a measure of total support (see the comments after the statement of \cref{p.dominationmodelocking}). 
\end{remark}
%
%
 
\subsubsection{Rotation numbers of perturbations in differentiable dynamics}\label{s.generalrelrot}

%
%
%
%
%
%

Let $K$ be a hyperbolic set of a diffeomorphism $f$ and let $\mu$ be an $f$ invariant measure supported in $K$. Assume that $Df$ preserves orientation on an oriented\footnote{The bundles we  consider in this paper are all continuous. An oriented bundle means for us a bundle endowed with a continuous orientation.} $2$-dimensional bundle $E_f$ of a dominated splitting over $K$. Take a simply connected neighborhood $\cU$ of $f$ on which the dominated splitting admits a continuation and let $E\colon g\mapsto E_g$ be the continuation of $E_f$. 
We prove in \cref{s.varyingbundles} that any $g\in \cU$ has a well-defined {\em rotation number relative to $f$ above $\mu$ along  $E$ } 
$$\rho_{f,E}(\mu,g)\in \RR$$
which counts the integrated rotation number of $Dg$ relative to $Df$ along $E_g$. 
This function depends continuously on $g$ for the $C^1$-topology and on $\mu$ for the weak-* topology.

For linear cocycles,  no hyperbolicity assumption is needed. Given a vector bundle $V$ with compact base $X$ and some homeomorphism $T\colon X\to X$, we consider the set  of continuous linear cocycles fibering on 
$T$ and endowed with the usual topology (see \cref{s.defsmoothbundles}).  
Let $\cA$ be a cocycle that preserves orientation on an oriented $2$-dimensional bundle $E_\cA$ of a dominated splitting.  
Take a simply connected neighborhood $\cU$ of $\cA$ on which there is a continuation $E\colon \cB\mapsto E_\cB$ of $E_\cA$. 

Then for any $T$-invariant measure $\mu$, any linear cocycle $\cB\in \cU$ admits  a {\em rotation number relative to $\cA$ above $\mu$  along $E$} 
$$\rho_{\cA,E}(\mu,\cB)\in \RR.$$
Again, that rotation number is continuous.

\begin{remark} The function $\rho_{*,E}(\mu,\cdot)$ depends on the choice of $\cU$, however its germ at $*$ does not, and it still contains all relevant information for our purpose. 
Hence we will often omit $\cU$ and identify the relative rotation number to its germ at $*$.
\end{remark}

\begin{remark}
Considering the finest dominated splitting and the tuple $\underline{E}=(E_1,\ldots, E_k)$ of $2$-dimensional bundles on which an orientation is preserved, we may consider the compound relative rotation number $\rho_{*,\underline{E}}\colon \cU \to \RR^k$.
\end{remark}

\subsection{Structure of the paper}
\cref{s.domsplitdef} states our results on perturbative differentiable dynamics. 
\cref{s.rotmorphism.modelock} defines precisely rotation numbers for paths of skew-products and gives some of its properties.
 \cref{s.varyingbundles} extends those results to 2-dimensional bundles of dominated splittings.
We use it in \cref{s.mainth} to prove the main results on linear cocycles. 
\cref{s.restofproofs} is there for the sake of completeness and essentially contains short proofs of the $C^0$-versions of the conjectures in this paper.
 \cref{a.avilakrikorian} recalls how Avila and Krikorian built a 'variation of the rotation number' for $\SL(2,\RR)$-cocycles, we notice that it coincides with ours up to a factor $-2$.

\section{Definitions and statement of results}\label{s.domsplitdef}
 We postpone the formal definitions of rotation numbers and mode-locking to \cref{s.rotmorphism.modelock}. 

\subsection{Smooth vector bundles and linear cocycles}\label{s.defsmoothbundles} Let $(V,\|.\|)$ be a continuous Euclidean vector bundle over a compact metric space $X$. Denote by $V_x$ the fiber above the point $x\in X$.
In this paper, a (maximal rank) {\em linear cocycle} on $V$ is a continuous map $\cA\colon V\to V$ fibering over a continuous transformation $T\colon X\to X$ such that the restriction of $\cA$ to the fiber $V_x$ is an isomorphism $A(x)\colon V_x\to V_{Tx}$, for all $x\in X$. In the case where the bundle is trivial, that is, $V=X\times B$ for a vector bundle $B$, we may regard $\cA$ as a pair $(T,A)$ where $A\colon X\to \cL(B)$ is a continuous map from $X$ to the space of isomorphisms $\cL(B)$ of $B$. 

If $X$ is a subset of a smooth manifold $M$, then $V$ is called {\em smooth} if it is the restriction of a smooth bundle $\widetilde{V}$ with base some open set $Y\subset M$ containing $X$. Define the space $LC^r_{V,\Id_X}$ of $C^r$-cocycles fibering over $\Id_X$ as the set of $C^r$-sections of the smooth bundle
$$\bigsqcup_{x\in X} GL(V_x)=V\otimes_X V^{*},$$
 endowed with its natural $C^r$-topology. 
 
 %
%
%
%
%
%

\subsection{$C^r$-generic cocycles over subshifts .}\label{s.simplespectrum} \label{s.complexdominationdichotconje}
Recall that the (two-sided) {\em full shift} over alphabet $\{1,\ldots,n\}$ is the map 
\[\begin{array}{rlcl}
\sigma\colon& \{1,\ldots,n\}^\ZZ&\to& \{1,\ldots,n\}^\ZZ\\
&(\omega_i)_{i\in \ZZ}&\mapsto&(\omega_{i+1})_{i\in \ZZ}
\end{array}\]

A {\em subshift of finite type}, or {\em SFT}, is the restriction of that full shift to an invariant set $\Omega_\tA$ defined as follows: the {\em transition matrix} $\tA=(a_{ij})_{1\leq i,j\leq n}$ is a matrix with coefficients in $\{0,1\}$ such that each line contains at least one $1$, and $\Omega_\tA$ is the set of words $\omega=(\omega_i)_i$ such that $a_{\omega_i\omega_{i+1}}=1$ for all $i$. A continuous transformation $T\colon K \to K$  of a metric space $K$ is called {\em transitive} if there exists $x\in K$ such that the positive orbit of $x$ is dense in $K$.

A transformation $T\colon X\to X$ is a {\em factor} of a transformation $S\colon Y\to Y$ if there exists a continuous surjection $\pi\colon Y\to X$ such that the following diagram commutes:
 $$\xymatrix{Y\ar[rr]^S \ar[d]^\pi && Y\ar[d]^\pi \\
            X\ar[rr]^{T} && X}$$ 

\begin{remark}\label{r.oginoga}
A transitive one-sided SFT (defined as above replacing $\ZZ$ by $\NN$) is a factor of the two-sided SFT with same transition matrix.
\end{remark}

In the following we consider for the base dynamics continuous transformations $T$ that are factors of transitive SFTs. 
Let us recall a classical family of such dynamics. 

A {\em hyperbolic set} of a diffeomorphism $f$ of a manifold $M$ is a compact $f$-invariant set $K$ such that there is a  $Df$-invariant splitting $TM_{|M}=E^s \oplus E^u$ into continuous bundles such that for all nonzero vectors $u,v\in E^s,E^u$, the norms of the iterates $Df^n(u),Df^{-n}(v)$ tend exponentially fast to $0$ and $+\infty$, respectively.
That set $K$ is a {\em basic} hyperbolic set if moreover it is
\begin{itemize}
\item transitive,
\item {\em locally maximal}: there is an open set $U$ such that $K=\displaystyle{\bigcap_{n\in \ZZ}} f^nU$.
\end{itemize}
 The dynamics on a basic hyperbolic set is a factor of a transitive SFT~\cite[p.594, Proposition 18.7.7]{KH}. 
 Fisher actually proved in \cite{F1,F2} that for any neighborhood $V$ of transitive hyperbolic set $K$, there exists a hyperbolic set $K\subset L\subset V$ such that the restriction $f_{|L}$ is a factor of a transitive SFT.

We say that a linear cocycle has {\em simple Lyapunov spectrum} at a $T$-periodic point $x$ if the eigenvalues (counted with multiplicity) of the first return linear map on the fiber $V_x$ have pairwise distinct moduli.

\begin{theorem}\label{t.cocyclewithsimpleLyap}
Let  $T\colon X\to X$ be a (not necessarily invertible) factor of a transitive SFT with $X$ infinite and let $V$ be a smooth bundle over $X$. For all $r\geq 0$, there is an open and dense subset $\cO\subset LC^r_{V,T}$ of linear cocycles  of class $C^r$ that have simple Lyapunov spectrum at some $T$-periodic point. 
\end{theorem}

That is, for any $\cA\in \cO$, there is a periodic point $x\in X$ along which the first return linear map has eigenvalues with pairwise distinct moduli.
Such a result is folklore for basic hyperbolic sets of diffeomorphisms (see \cref{f.simplespectrumdiffeos}) thanks to the non-elementary Sternberg linearization theorem. 
That tool does not exist for linear cocycles. As said in \cite{BV}, some rotation theory is needed. We explain it in \cref{s.rotmorphism.modelock,s.varyingbundles}.

\begin{definition}
Given a linear cocycle $\cA\colon V\to V$ fibering over a continuous transformation $T\colon X\to X$, a {\em dominated splitting} is a  splitting $$V=E_1\oplus\ldots \oplus E_k$$ into $\cA$-invariant vector subbundles of constant dimension, such that  for some $N\in \NN$, for all $x\in X$ and all unit vectors $u\in E_{i,x}$,  $v\in E_{i+1,x}$ above $x$ and in consecutive bundles, one has
$$\|\cA^N u\|<\frac{1}{2}\|\cA^N v\|.$$
\end{definition}

We say that a dominated splitting into 2 bundles $V=E_1\oplus E_2$ {\em has index $i$} if $\dim E_1=i$.
Lyapunov spectrum simplicity at a periodic point corresponds to the existence of dominated splittings of all indices (equivalently of a dominated splitting into line bundles) along that periodic point.
A robust obstruction to the existence of a dominated splitting of index $i$ is the following object:

\begin{definition} Let $1\leq i<d$. A periodic point $x$ is {\em $i$-elliptic} if its Lyapunov exponents $\lambda_1\leq \ldots \leq \lambda_d$  at $x$  satsify:
\begin{itemize}
\item $\lambda_i=\lambda_j$ if and only if $i=j+1$,
\item the eigenvalues corresponding to $\lambda_i$ and  $\lambda_{i+1}$ are complex (non-real) and conjugate.
\end{itemize}
\end{definition}

We say that a linear cocycle or a diffeomorphism {\em satisfies the ellipticity/domination dichotomy} if for each index $0< i<d$, where $d=\dim V$ or $\dim M$, we have the following sharp dichotomy:
\begin{itemize}
\item either there is a global dominated splitting of index $i$,
\item or there is an $i$-elliptic periodic point.
\end{itemize}

\begin{remark}
In \cref{t.cocyclewithsimpleLyap} we do not need to assume that the base transformation $T$ is invertible. The following conjectures and results about linear cocycles however only deal with cocycles fibering over a homeomorphism $T$.
\end{remark}

\begin{conjecture}[ellipticity/domination dichotomy for cocycles]\label{c.complexdominationdichotconje}
Let a homeomorphism $T\colon X\to X$ be a factor of a transitive SFT and let $V$ be a smooth vector bundle on $X$. 
For each $r\geq 0$, there is an open and dense subset $\cO\subset LC^r_{V,T}$ of linear cocycles that satisfy the ellipticity/domination dichotomy.
\end{conjecture}

The assumption on $T$ is one of an abundance of periodic point. This may compared with the hyperbolicity/bounded cocycles dichotomies of Avila, Bochi, Damanik~\cite{ABD} under diametrically opposite conditions: if $T$ is minimal on $X$ finite dimensional, then is an open and dense subset of continuous $SL(2,\RR)$-cocycles that are either hyperbolic or (continuously) cohomologous to an $\SO(2,\RR)$-cocycle.

\begin{remark}
If $T$ is an irrational rotation of the circle $\TT$, ellipticity can obviously not happen. And if a linear cocycle $\cA$ on  $V=\TT\times \RR^2$ fibers above $T$ and is not isotopic to $\Id$, then the  induced projective cocycle \[\PP\cA\colon \PP V:=\TT\times \PP(\RR^2)\to \PP V,\]
acts parabolically on the homology group $H_1(\PP V, \RR)$. In particular, it cannot leave invariant any section $s\colon X\to  \PP(\RR^2)$, that is, $\cA$ cannot leave invariant any $1$-dimensional bundle. Thus we have open sets of linear cocycles fibering above an irrational rotation that do not satisfy the ellipticity/domination dichotomy.
\end{remark}

It is tempting to weaken the assumptions on the homeomorphism $T$ in \cref{c.complexdominationdichotconje} to that of a density of periodic points. This fails however to take into account the possibility of orientation inversion phenomena, as seen in the following counter example:

\begin{remark}
The small $C^0$-perturbations in $LC_{V,T}$ of the following cocycle have both no dominated splitting and no elliptic periodic point: let $\cA=(T,A)$ be the linear cocycle on $V=[-1,1]\times \RR^2$ defined by $T=\Id_{[-1,1]}$ and 
$$A(x)=\left(\begin{array}{cc}
e^{-x}&0\\
0& -e^x
\end{array}\right).
$$ 
\end{remark}

A similar conjecture may be stated for diffeomorphisms. Indeed, as a consequence of the shadowing lemma, a hyperbolic set $K_f$ of a diffeomorphism admits a {\em continuation} $g\in \cU\mapsto K_g$ on a $C^1$-neighborhood $\cU$ of $f$. More precisely, there is a continuous map $h\colon K_f\times \cU\to M$ such that for all $g\in \cU$, the set $K_g:=h(K_f\times \{g\})$ is a hyperbolic set for $g$ and \[x\in K_f\mapsto h(x,g)\in K_g\] is a homeomorphism that conjugates the restricted dynamics $g_{|K_g}$ and $f_{|K_f}$. That continuation is uniquely defined on each connected component of $\cU$.

\begin{conjecture}[ellipticity/domination dichotomy for diffeomorphisms]\label{c.complexdominationdichotconjediffeo}
Let $r\geq 1$. Fix a basic hyperbolic set $K_f$ of a $C^r$-diffeomorphism $f$ and a $C^r$-neighborhood $\cU$ of $f$ on which a continuation $g\mapsto K_g$ is defined. 

Then there exists an open and ($C^r$) dense subset $\cO\subset \cU$ of diffeomorphisms $g$ such that the cocycles $Dg_{|K_g}$ satisfy the ellipticity/domination dichotomy.
\end{conjecture}

We prove \cref{c.complexdominationdichotconje} for $r=0$ through classical techniques (see \cref{s.abcequivalence}). \cref{c.complexdominationdichotconjediffeo} for $r=1$ is then a straightforward consequence using the Franks' lemma~\cite{Fr}. See for instance~\cite{BCDG} for similar constructions. 

The real difficulty starts at higher regularities due to the lack of perturbative tools. Using the properties of rotation numbers we are nevertheless able to prove some cases, as in the theorem below. We recall first that a linear cocycle always admits a finest dominated splitting (that is, a dominated splitting which splits any other dominated splitting).

\begin{theorem}\label{t.growuptheorem} 
Let $r\geq 0$. Let $T\colon X\to X$ be a homeomorphism admitting a dense set of periodic points and let $V$ be a smooth vector bundle on $X$. 
Let $\cA$ be a $C^r$ linear cocycle on $V$ fibering over $T$. Let $E_1,\ldots, E_\ell$ be the $2$-dimensional bundles of the finest dominated splitting on which some orientation is preserved. 

Then in any $C^r$-neighborhood of $\cA$ there is a cocycle that has an elliptic point by restriction to the continuation of each of these bundles.
\end{theorem}

This straightforwardly implies the following:

\begin{corollary}\label{t.complexdominationdichotconje}
Let $r\geq 0$.  Let $T\colon X\to X$ be a homeomorphism admitting a dense set of periodic points and let $V$ be a smooth vector bundle on $X$. 
There is a $C^r$-open and dense subset $\cO\subset LC_{V,T}$ of linear cocycles $\cA$ such that if each bundle of its finest dominated splitting
\begin{itemize}
\item either has dimension $1$,
\item or has dimension $2$ and $\cA$ preserves some orientation on it,
\end{itemize}
then $\cA$ satisfies the ellipticity/domination dichotomy. \end{corollary}

In a work in progress~\cite{BBG}, we prove the same results for totally disconnected basic hyperbolic sets of diffeomorphisms, making some steps towards \cref{c.complexdominationdichotconjediffeo}.

\subsection{Dominated splittings and structural stability.}\label{s.structstab}
Let $r\geq 0$. Two issues with the ellipticity/domination dichotomy are that:
\begin{itemize}
\item it seems very difficult to prove in all generality for $r>0$,
\item it is not true for all dynamics. 
\end{itemize}
Instead, we may substitute to ellipticity a natural notion of structural stability for cocycles. 
The $C^r$-topology on the space of linear cocycles on $V$ naturally induces a $C^r$-topology on the space of projective cocycles.

\begin{definition}
Fix a homeomorphism $T\colon X\to X$.
A projective $C^r$-cocycle $\PP \cA\colon \PP V\to \PP V$ fibering over $T$ is called {\em structurally $C^r$-stable} if for all $\epsilon>0$ it admits a $C^r$-neighborhood $U_\epsilon$ among cocycles fibering on $T$ such that:

Seeing them as fibered homeomorphisms of $\PP V$, any $\PP\cB\in U_\epsilon$ is cohomologous to $\PP\cA$ through a homeomorphism that is $\epsilon$-close to ${\Id_{\PP V}}$, that is, there exists a homeomorphism $h$ of $\PP V$ that fibers on $\Id_X$, that is $\epsilon$-close to $\Id_{\PP V}$, and such that 
$$h\circ \PP \cA=\PP\cB \circ h.$$
\end{definition}
We say that a linear cocycle satisfies {\em $C^r$-robustly} some property if that property holds for any sufficiently $C^r$-close linear cocycle. 
Consider the three following statements:
\begin{itemize}
\item[(a)] The cocycle $\cA$ has dominated splittings of all indices $0< i<n$,\footnote{If the non-wandering set of $T$ is the whole base $X$, then the existence of dominated splittings of all indices is equivalent to projective hyperbolicity in the following sense: the chain-recurrent set of $\PP\cA$ is included in the union of projective bundles $\cup \PP E_i$, where the $E_i$-s are the bundles of the finest dominated splitting. Consider then the linear cocycle $T \PP\cA$ 
$$
\begin{array}{ccc}
T \PP V & \xrightarrow{T \PP \cA} & T\PP V\\
\downarrow & & \downarrow\\
\PP V & \xrightarrow{\PP \cA} & \PP V
\end{array}$$
induced on the tangent projective bundle $T\PP V=\cup_{v\in \PP V}T_v{\PP V_x}$ where $x\in X$ is the base point  on which $v$ projects. It is quite easy to check that the cocycle  
$T \PP\cA$ restricted to each compact set $ \PP E_i$  of  $ \PP V$ admits a hyperbolic decomposition if and only if $E_i$ has dimension $1$. 

In general, when the non-wandering set of $T$ is not necessarily $X$, domination of all indices could be seen as projective hyperbolicity and a strong transversality condition.}
\item[(b)] The projective cocycle $\PP\cA$ is structurally $C^r$-stable,
\item[(c)] The cocycle $\cA$ has $C^r$-robustly no elliptic point.
\end{itemize}
\cref{c.complexdominationdichotconje} implies that if the base homeomorphism $T$ is a factor of a transitive SFT, then we have the equivalence
$(a)\Leftrightarrow (c).$ In \cref{s.restofproofs}, using standard tools we obtain:
\begin{proposition}\label{p.abc} For each $r\geq 0$ we have the implications  $(a)\Rightarrow (b)\Rightarrow (c).$
\end{proposition}
We also show in \cref{s.abcequivalence} the following converse implications: if $r=0$ then $(a)\Leftarrow (b)$ and if moreover $T$ is a factor of a transitive SFT, then $(a)\Leftarrow (c)$. The real difficulty starts with higher regularity. 
\begin{conjecture}{\bf (Stability):}\label{c.ba}
For each $r\geq 0$ we have the equivalence
$(a)\Leftrightarrow (b).$
\end{conjecture}
We make a little step towards \cref{c.ba} using properties of rotation numbers:
%
%
%
%

\begin{theorem}\label{t.projstructstab}
Let $\cA$ be a linear cocycle such that: 
\begin{itemize}
\item it fibers over a homeomorphism $T$ admitting a finite invariant measure of total support,
\item the bundles of its finest dominated splitting have all dimension $\leq 2$,
\item $\cA$ preserves some continuous orientation on each of the $2$-dimensional bundles of that splitting.
\end{itemize}
Then we have the equivalence $(a)\Leftrightarrow (b)$, that is, $\PP\cA$ is structurally $C^r$-stable if and only if it has dominated splittings of all indices.
\end{theorem}

\begin{remark}
It seems to us that a lot of the difficulties of the conjectures are already contained in the case of $\SL^{\pm}(2,\RR)$-cocycles, where $\SL^{\pm}(2,\RR)$ is the set of real matrices of determinant $\pm 1$. This should be the next problem to focus on, before tackling higher dimensions.
\end{remark}

\begin{question} The implication $(a)\Rightarrow (b)$ of \cref{p.abc} might be a particular case of a generalization the theorems of Robbin and Robinson~\cite{Rob,Robb} to fibered diffeomorphisms. 

Here is a tentative formulation: let a fibered homeomorphism $f$ on a smooth fiber bundle $V$ such that $f$ is a diffeomorphism by restriction to each fiber, and the linear cocycle \[Df\colon TV:=\bigsqcup_{x\in X} TV_x\to TV\] is continuous. Call it {\em hyperbolic} if, by restriction to each chain-recurrent set of $f$, the cocycle $Df$ is hyperbolic. Is any such $f$ structurally stable? More precisely, is any $C^1$-perturbation $g$ of $f$ along the fibers cohomologous to $f$, that is, conjugate to $f$ by a homeomorphism of $V$ that fixes each fiber?
\end{question}

\subsection{Ideas of the proofs}

\begin{proof}[Idea of the proof of \cref{t.cocyclewithsimpleLyap}]If we have a pair of non-real conjugate eigenvalues at some periodic point $x$, then a small perturbation might change their arguments slightly but not turn them real. The point is that, above a periodic measure (an invariant measure supported by a periodic point), our rotation number corresponds to the change of argument of the eigenvalues, multiplied by the period. Hence, using the continuous dependence of the rotation number on the measure, by taking a periodic measure with large period close enough to our initial periodic measure and using the continuity of the rotation number, we get a large change of argument along the new large period periodic point, which means that we go at some point through a pair of  real eigenvalues. A further smooth perturbation separates the moduli of those eigenvalues.
\end{proof}

Given a linear cocycle $\cA$ on a linear bundle $V$ denote by $\PP\cA$ its projectivization, that is, the induced map of the projective bundle $\PP V$. The following extends a result Avila, Bochi and Damanik~\cite{ABD} to the case of relative rotation numbers:

\begin{proposition}\label{p.dominationmodelocking}
Let a linear cocycle $\cA\in LC_{V,T}$ preserve some continuous orientation on a $2$-dimensional subbundle $E$. Let $\mu$ be a $T$-invariant measure. Then the following are equivalent:
\begin{itemize}
\item[$(i)$] the projectivization $\PP\cA_{|E}$ is mode-locked over $\mu$,
\item[$(ii)$] the restriction $\cA_{|E}$ admits a dominated splitting over the support $\supp(\mu)$.
\end{itemize}
\end{proposition}

We give a proof in \cref{a.p.dominationmodelocking}.
Here mode-locking is seen among the set of skew products $Sk^+_{\PP E,T}$.
Note that the second item is equivalent to the existence of a trapping strip, which motivated \cref{r.trappingstri}. Indeed, taking an adapted metric~\cite{G1} for the dominated splitting, one easily builds a trapping strip. Conversely, if there is a trapping strip, then \cite[Theorem B]{BG1} implies that there is a dominated splitting. 

\begin{proof}[Ideas of the proofs of \cref{t.complexdominationdichotconje} and \cref{t.projstructstab}]
Mode-unlocking for $\PP(\cA_{|E})$ above a measure $\mu$ means by definition that there is a perturbation $g$ in $Sk^+_{\PP E,T}$ with a non-trivial relative rotation number above $\mu$. We actually prove that mode-unlocking implies that there are smooth perturbations of $\cA$ with non-zero rotation numbers along the continuation of $E$. This implies in particular that there is no projective structural $C^r$-stability.
Moreover, if $\mu$ is accumulated by periodic measures (which is the case in transitive subshifts of finite type), then by continuity of the rotation number those periodic measures will see that non-trivial rotation number. In other words the argument of the eigenvalues along the continuation of $E$ have to change: smooth perturbations will create non-real eigenvalues. 
\end{proof}

%
%
%

\subsection{Structure of the rest of the paper}
In \cref{s.rotmorphism.modelock}, we define precisely rotation numbers for paths of skew-products, we give some of its properties and we explain a corresponding notion of mode-locking.
In \cref{s.varyingbundles}, we extend those results to 2-dimensional bundles of dominated splittings.
In \cref{s.mainth} we use rotation numbers to prove \cref{t.cocyclewithsimpleLyap,t.growuptheorem,t.projstructstab}. 

For completeness, in \cref{s.restofproofs} we use classical arguments to prove the $C^0$ versions of the conjectures above as well as the proof of \cref{p.abc}. Finally in \cref{a.avilakrikorian} we recall how Avila and Krikorian built a 'variation of the rotation number' for $\SL(2,\RR)$-cocycles, and we notice that it coincides with ours up to a factor $-2$.

\section{Rotation numbers and mode-locking for a fixed circle bundle}\label{s.rotmorphism.modelock}
In this section we fix a continuous (not necessarily invertible) transformation $T\colon X\to X$ of a compact metric space $X$.
Fix an oriented circle bundle $\EE$ with compact base $X$. We work in the set $Sk^+_{\EE,T}$ of skew-products on $\EE$ fibering on $T$, endowed with the topology of uniform convergence.

Endow $\EE$ with a metric that gives length $1$ each circle. This gives a metric on the universal cover of each fiber. 
Given real numbers $a\leq b$ and a path $\gamma=(\gamma_t)_{t\in [a,b]}$ in some fiber $\EE_x$, let  $\widetilde{\gamma}$ be a lift to the universal cover $\tilde{\EE}_{x}$. Define the {\em winding number}  $w{(\gamma)}$ as the algebraic length
of the oriented segment $[\widetilde{\gamma}_a,\widetilde{\gamma}_b]$. It does not depend on the choice of the lift.

Given a path $\phi=\bigl(f_t\bigr)_{t\in [a,b]}$ of skew-products in $Sk^+_{\EE,T}$ and a point $u\in \EE_x$ in the fiber above $x$, denote by $\phi u$ the path $\bigl(f_tu\bigr)_{t\in [0,1]}$. Denote by $\phi^{(n)}$ the path of $n$-th iterates $f_t^n=f_t\circ \ldots \circ f_t$ in $Sk^+_{\EE,T^n}$. Since for all $u,u'\in \EE_x$,
\begin{align}
| w(\phi^{(n)}  u)- w(\phi^{(n)} u')|<1.\label{e.winding}
\end{align}
we can unambiguously introduce the following:

\begin{definition}\label{d.rotnumberpoint}
The {\em rotation number} of the path $\phi$ at $x\in X$ is, whenever it exists, the asymptotic growth rate of winding numbers through iteration: 
\begin{align*}\rho(x,\phi)=\lim_{n\to +\infty}\frac{1}{n} w(\phi^{(n)} u),
\end{align*}
where $u$ is any vector in $\EE_x$ (it does not depend then on the choice of that $u$).
\end{definition}

The existence and the value of this limit does not depend on the choice of the metric on $\EE$.
Let $\cM(T)$ be the vector space of $T$-invariant finite Borel signed measures on $X$, endowed with the weak-* topology and let $P_1(Sk^+_{\EE,T})$ be the set  of paths of $Sk^+_{\EE,T}$ endowed with the topology of uniform convergence. Endow $\cM(T)\times P_1(Sk^+_{\EE,T})$ with the product topology.

\begin{proposition}[Existence and properties of rotation numbers]\label{p.rotationmorphism}
Consider a path $\phi$ in $Sk_{\EE,T}^+$ and a measure $\mu\in \cM(T)$. The rotation number $\rho(\cdot, \phi)$ is defined $\mu$-a.e., bounded, $T$-invariant, measurable and only depends on the homotopy class of $\phi$ with fixed extremities. Moreover,
$$\rho(\mu,\phi):= \int_X \rho(\cdot,\phi)d\mu$$
 defines a continuous map $\rho\colon \cM(T)\times P_1(Sk^+_{\EE,T})\to \RR$.
\end{proposition}

We call the map  $\rho\colon \cM(T)\times P_1(Sk^+_{\EE,T})\to \RR$ the {\em rotation map} of the pair $(T,\EE)$ and the real number $\rho(\mu,\phi)$ the {\em rotation number} of  the path $\phi$ above the measure $\mu$.

\begin{remark}
The rotation map $\rho$ satisfies then the following additional properties: 
\begin{itemize}
\item it is by construction linear in variable $\mu$,
\item it is additive with respect to concatenation:  $\rho(\mu,\phi*\psi)=\rho(\mu,\phi)+\rho(\mu,\psi)$, for any paths $\phi,\psi$ such that $\phi*\psi$ is defined, as a consequence of the same property for winding numbers. 
\end{itemize}
\end{remark}

\begin{remark}
Suppose that $\EE=X\times \SS^1$, the base $X$ is connected, and $f$ is isotopic to identity through a path $\phi$ in $Sk^+_{\EE,T}$ joining $(T,\Id)$ to $f$. Then $\rho(\mu,\phi)\mod \ZZ$ coincides with the fibered rotation number $\rho_\mu(f)$ defined by Herman~\cite{H,JS}.
\end{remark}


%

\begin{proof}[Proof of \cref{p.rotationmorphism}]

Choose a metric on $\EE$ that gives length $1$ each circle and define the 
sequence of continuous functions
\begin{align*}
\tau_n(x,\phi)&=\max_{u\in \EE_x}w(\phi^{(n)} u).
\end{align*}
with respect to some metric on $\EE$. That sequence is subadditive, that is, 
$$\tau_{n+m}(x,\phi)\leq \tau_n(x,\phi)+\tau_m(T^nx,\phi).$$
The Kingman's subadditive ergodic theorem~\cite{Ki} gives that:
\begin{itemize}
\item for any $T$-invariant measure $\mu$, the sequence $\frac{1}{n}\tau_n(\cdot,\phi)$ converges $\mu$-a.e. to a measurable $T$-invariant map, which is $\rho(\cdot,\phi)$ as in \cref{d.rotnumberpoint}. That map is bounded since $\tau_1$ is.
\item $\rho(\mu,\phi)=\inf_{n>0} \frac{1}{n}\int\tau_n(\cdot,\phi)d\mu.$
\end{itemize}
In particular, the continuity of $\tau_n$ on $X\times P_1(Sk^+_{\EE,T})$ and the compactness of $X$ imply that 
$$(\mu,\phi)\mapsto \frac{1}{n}\int\tau_n(\cdot,\phi)d\mu$$
is continuous. Hence the upper-semicontinuity of
 $$\rho\colon \cM(T)\times P_1(Sk^+_{\EE,T})\to \RR.$$ Lower-semicontinuity is obtained symmetrically by considering the superadditive  sequence of functions  
\begin{align*}
\sigma_n(x,\phi)=\min_{u\in \EE_x}w(\phi^{(n)} u).
\end{align*}
Invariance by isotopy with fixed extremities comes from the fact that the winding numbers $w(\phi^{(n)} u)$ are themselves invariant by isotopies with fixed extremities.
This ends the proof of \cref{p.rotationmorphism}.
\end{proof}

Two skew-products $f,g\in Sk^+_{\EE,T}$ are (continuously) {\em cohomologous} if they are conjugate by a homeomorphism in $Sk^+_{\EE,\Id_X}$. Denote by  $[f]$  the cohomology class of $f$.

\begin{proposition}[the rotation number is a cohomology invariant]\label{p.conjugacyinvariant}
If two paths  $\phi=\bigl(f_t\bigr)$ and $\psi=\bigl(g_t\bigr)$ coincide in cohomology, that is, $[f_t]=[g_t]$ for all $t$, then their rotation numbers coincide:
for all $x\in X$, if one of the rotation numbers $\rho(x,\phi)$ and $\rho(x,\psi)$ exists, then both do and they coincide.
\end{proposition}

\begin{proof}[Proof of \cref{p.conjugacyinvariant}] 
Given two continuous transformations $T,T'\colon X\to X$ and two paths $\phi=\bigl(f_t\bigr)$ and $\psi=\bigl(g_t\bigr)$ in $Sk^+_{\EE,T}$ and $Sk^+_{\EE,T'}$, respectively, denote by $\psi\circ \phi$ the path of composed maps $\bigl(g_t\circ f_t\bigr)$ in $Sk^+_{\EE,T'\circ T}$. We put a metric on $\EE$ that gives length $1$ and diameter $1/2$ to each circle.
The path $\psi\circ \phi$ is isotopic with fixed extremities to the concatenation $(g_1\circ  \phi) * (\psi \circ f_0)$, so that  for any $u\in \EE$ we have 
\begin{align*}
w(\psi\circ \phi \;u)&=w\Bigl[ (g_1\circ  \phi) * (\psi \circ f_0)\; u\Bigr]\\
&=w(g_1\circ  \phi \;u)+w(\psi \circ f_0\;u).
\end{align*}
Moreover
\begin{align*}
\left|w(g_1 \circ \phi \;u)-w(\phi u )\right|&<1\\
\left|w(\psi \circ f_0\; u)-w(\psi u)\right|&<1
\end{align*}
 therefore 
\begin{align}
\left|w(\psi \circ \phi \; u)-w(\psi  u)-w(\phi u)\right|&<2.\label{e.wind}
\end{align}
This allows to prove this particular case of \cref{p.conjugacyinvariant}:

\begin{claim}\label{c.eowfwihfe} If there exists a path  $\zeta$ in $Sk^+_{\EE,\Id_X}$ such that 
\begin{align*}
\zeta\circ \phi=\psi\circ \zeta,
\end{align*}
then the rotation numbers of $\phi$ and $\psi$ coincide.
\end{claim}

\begin{proof}[Proof of \cref{c.eowfwihfe}] For all $n\in \NN$, we have $\zeta\circ \phi^{(n)}=\psi^{(n)}\circ \zeta$. Applying \cref{e.wind} to each member of the equality $\zeta\circ \phi^{(n)}\;u=\psi^{(n)}\circ \zeta\;u$ we deduce
\begin{align*}
|w(\psi^{(n)} u)-w(\phi^{(n)} u)|&<2w(\zeta u)+4
\end{align*}
for all $u\in \EE$. In particular, if one of the rotation numbers $\rho(x,\phi)$ and $\rho(x,\psi)$ exists, then both do and they coincide.
\end{proof}

We now prove \cref{p.conjugacyinvariant} in full generality. Consider two paths  $\phi=\bigl(f_t\bigr)_{t\in [a,b]}$ and $\psi=\bigl(g_t\bigr)_{t\in [a,b]}$ such that for all $t\in [a,b]$ there exists $h_t\in Sk^+_{\EE,\Id_X}$ such that $f_t\circ h_t=h_t\circ g_t$. 
The difficulty is that $h_t$ cannot in general be chosen to depend continuously on $t$. It is  quite easy to build examples of this impossibility.
We endow the sets of maps from $\EE$ to $\EE$ with the maximum distance: 
$$d(f,g)=\max_{u\in \EE\atop \iota\in \{-1,1\}} d(f^\iota u,g^\iota u).$$
Given $\delta>0$, let $H_\delta$ be the open set of homeomorphisms in $Sk^+_{\EE,\Id_X}$ such that $d(u,v)\leq \delta$ implies $d(h^\iota u,h^\iota v)< 1/4$, for $\iota=\pm 1$. 
By a compactness of $[a,b]\times \EE$, there exists $\eta_\delta>0$ such that
\begin{itemize}
\item for any $r,s\in [a,b]$ such that $|r-s|<\eta_\delta$,
\item for any $h,h'\in H_\delta$ such that $d(h,h')<\eta_\delta$,
\end{itemize}
 we have both
\[
 d(f_r,f_s)<1/4\quad \mbox{ and }\quad
 d(h\circ g_r\circ h^{-1},h'\circ g_s\circ h'^{-1})<1/4.\]
The sets $H_\delta$ are locally arc-connected and separable and $Sk^+_{\EE,\Id_X}=\cup_{k\in \NN} H_{1/k}$, therefore $Sk^+_{\EE,\Id_X}$ is a countable union of open arc-connected sets $\cU_n$, each of which is a subset of some $H_{\delta(n)}$ and has radius $<\eta_{\delta(n)}$.

 Let $I_n\subset [a,b]$ be the subset of parameters $t$ such that there exists $h\in \cU_n$ that conjugates $f_t$ and $g_t$. 
 
 \begin{claim}\label{c.efWEFWEwne}For all $r,s\in I_n$ satisfying $|r-s|<\eta_{\delta(n)}$, the path $(f_t)_{t\in [r,s]}$ has same rotation number as $(g_t)_{t\in [r,s]}$.\label{c.rfjeprj}
 \end{claim}
 
 \begin{proof}[Proof of \cref{c.rfjeprj}]
 Let $h_r,h_s\in \cU_n$ such that  $f_*=h_*\circ g_*\circ h_*^{-1}$ for $*=r,s$. Join them by a path $(h_t)_{t\in[r,s]}$ in $\cU_n$. Then by definition of  $\eta_{\delta(n)}$ and $\cU_n$, the paths $(f_t)_{t\in[r,s]}$ and $(h_t\circ g_t\circ h_t^{-1})_{t\in [r,s]}$ have diameter $<1/4$. Concatenating the first path and second path backwards gives a loop in $Sk^+_{\EE,T}$ with diameter $<1/2$. As a loop, its winding numbers at each $u\in \EE$ are integers. The diameter of the loop, which is less than the diameter of the circles of $\EE$, forces then the winding numbers to be $0$. The winding numbers of the iterated loops and therefore the rotation number is also 0 at each $x\in X$\footnote{One could actually prove that the loop is isotopically trivial in $Sk^+_{\EE,T}$.}. Thus $(f_t)_{t\in [r,s]}$ has same rotation number as $(h_t\circ g_t\circ h_t^{-1})_{t\in [r,s]}$ and therefore as $(g_t)_{t\in [r,s]}$ by  \cref{c.eowfwihfe}. 
 \end{proof}

\begin{claim}\label{c.eorfiaraa}
The interval $[a,b]$ is covered a finite or countable set of segments $\cS$ whose interiors are pairwise disjoint and such that  the rotation numbers of $(f_t)$ and $(g_t)$ coincide along each $S\in \cS$.
\end{claim}

\begin{proof}
Define inductively a sequence of subsets $\Sigma_n\subset [a,b]$ for $n\geq -1$ as follows:
\begin{itemize}
\item let $\Sigma_{-1}=\emptyset$, 
\item let $\Sigma_n$ be the union of segments in $[a,b]\setminus \Sigma_{n-1}$ of length $<\eta_{\delta(n)}$ whose extremities lie in $I_n$.
\end{itemize}
We let the reader check inductively that each $\Sigma_n$ is a finite union of intervals. Then each $\Sigma_n$ is a union of finitely or countably many segments whose extremities lie in $I_n$, whose lengths are $<\eta_{\delta(n)}$, and whose interiors are pairwise disjoint.  We have $[a,b]=\cup_{N\in \NN} I_n$ therefore  the sets $\Sigma_n$ form a partition  of $[a,b]$.
On concludes with  \cref{c.rfjeprj}.
\end{proof}

We now prove that  $(f_t)_{t\in [a,b]}$ and $(g_t)_{t\in [a,b]}$ have same rotation numbers. This is less obvious than it could seem, one difficulty being that $(f_t)$ and $(g_t)$ having same rotation number along an interval $[r,s]$ does not a priori mean that this also holds along subintervals.
Let $\cI$ be the countable set of intervals in $[a,b]$ whose extremities are extremities of intervals of $\cS$. Let $\cJ$ be the set of intervals $J\in \cI$ such that for each $I\in \cI$ such that $I\subset J$, the rotation numbers of $(f_t)$ and $(g_t)$ coincide along $I$. Clearly $\cS\subset \cJ$.

By construction of $\cJ$, if an interval $J$  is a union of two intervals of $\cJ$, then it is a union $J=J_1\cup J_2$ of two intervals of $\cJ$ whose interiors are disjoint. Using additivity of the rotation number for concatenation, we easily get that the rotation numbers of  $(f_t)$ and $(g_t)$ coincide along $J$, and likewise for any $J'\in \cJ$ such that $J'\subset J$, therefore  $J\in \cJ$. Thus $\cJ$ is stable by finite connected unions. 

Likewise, if an interval $J$ is the union of a nested sequence $J_1\subset J_2\subset \ldots $ of $\cJ$, then the continuity of the rotation number implies that the rotation numbers coincide along any $J'\in \cJ$ such that $J'\subset J$, therefore  $J\in \cJ$. Thus $\cJ$ is stable by union of nested sequences.

Given $x\in [a,b]$, let  $J_x$ be the union of intervals $J\in \cJ$ that contain $x$. This is a countable union since $\cJ$ is countable. Since $\cJ$ is stable by finite connected unions, $J_x$ can be written as the union of a nested sequence in $\cJ$, hence $J_x\in \cJ$, that is $J_x$ is the maximal element of $\cJ$ that contains $x$. If $z$ is an extremity of $J_x$ then the connected union $J_x\cup J_z$ of two elements of $\cJ$ is in $\cJ$ and is equal to $J_x$ by maximality. This proves that $J_x$ contains its extremities: it is a segment. 
Since $\cJ\subset \cI$ is countable, the sets $J_x$ form a countable partition of $[a,b]$ into segments. This implies that $J_x=[a,b]$ for all $x$. In particular the rotation numbers of $(f_t)$ and $(g_t)$ coincide along $[a,b]$. This ends the proof of \cref{p.conjugacyinvariant}.
\end{proof}

Choose a simply connected neighborhood $\cU$ of $f$ in $Sk_{\EE,T}^+$. Then any $g\in \cU$ has a well-defined {\em  rotation number relative to $f$ in $\cU$} defined as follows:
$$\rho_f(\mu,g):=\rho(\mu,\phi)$$
where $\phi$ is a path in $\cU$ going from $f$ to $g$. 

\begin{remark}\label{r.germsk}
Another choice of $\cU$ gives another function $\rho'_f$ which coincides with $\rho_f$ on a neighborhood of $f$. In other words, the germ at $f$ of the function
$$\begin{cases}
\cU&\to \cL\bigl[\cM(T),\RR\bigr]\\
g&\mapsto \rho_f(\cdot,g)
\end{cases}$$ is well-defined independently of $\cU$. 
\end{remark}

The usual mode-locking notion \cite{BJ} naturally extends to our setting:

\begin{definition}
A skew-product $f\in Sk_{\EE,T}^+$ on $\EE$ is {\em mode-locked} over the $T$-invariant finite measure $\mu$ if the relative rotation number vanishes on a neighborhood of $f$. 
\end{definition}

Note that $f$ is mode-locked over $\mu$ if all paths in some neighborhood of $f$ have zero rotation number. We call the skew-product {\em lower semi-locked} (resp. {\em upper semi-locked}) over $\mu$ if the relative rotation number is $\geq 0$ (resp. $\leq 0$) on a neighborhood of $f$.
We call a path $\phi$ {\em strictly increasing} if, for any $u\in \EE$, the path $\phi u$ is strictly increasing with respect to the orientation on the fibers of $\EE$.

\begin{lemma}\label{l.modelocking.increasing}
A skew-product $f$ is not upper semi-locked over $\mu$ if and only if the strictly increasing paths $\phi=\bigl(f_t\bigr)_{t\in [0,1]}$ starting at $f=f_0$  have all strictly positive rotation number above $\mu$.
\end{lemma}

\begin{proof}
The converse implication is straightforward. Let us show the direct one. 
Take a strictly increasing path $\phi=\bigl(f_t\bigr)_{t\in [0,1]}$ starting at $f=f_0$. By compactness we find $0<\epsilon<1/2$ such that for all $u\in \EE$,
 $$\epsilon<w(\phi u).$$
Let $\cU$ be a simply connected $\epsilon$-neighborhood of $f$, and let $\psi$ be a path in $\cU$ from $f$ to some $g\in \cU$ with strictly positive rotation number relative to $f$, that is, $\rho_f(g,\mu)=\rho(\psi,\mu)> 0.$
Then, for all $u\in \EE$,
 $$w(\psi u)<\epsilon<w(\phi u).$$
By a monotonicity argument  $w(\psi^{(n)} u)<w(\phi^{(n)} u)$, for all $u\in \EE$ and $n\in\NN$. Therefore
$0<\rho(\psi,\mu)\leq \rho(\phi,\mu)$.
\end{proof}

Changing orientation on $\EE$, one deduces from this lemma that lower semi-locking is equivalent to all strictly decreasing paths starting at $f$ having strictly negative rotation numbers. In particular we get:

\begin{corollary}\label{c.modelocking.increasing}
A skew-product $f$ is not mode-locked over $\mu$ if and only if at least one of the two following statements is true:
\begin{itemize}
\item  the strictly increasing paths $\phi=\bigl(f_t\bigr)_{t\in [0,1]}$ starting at $f=f_0$ have strictly positive rotation number above $\mu$,
\item  the strictly decreasing paths starting at $f$ have strictly negative rotation number.
\end{itemize}
\end{corollary}

\begin{lemma}\label{l.rotationcriteria}
If $\phi=(f_t)$ is a strictly increasing path, then the following conditions are equivalent:
\begin{itemize}
\item[$(i)$]  the rotation number satisfies $\rho(\mu,\phi)>0$,
\item[$(ii)$]  for some $x\in \supp(\mu)$ and some $N\in \NN$, 
$$\sigma_N(x):=\min_{u\in E_x}w(\phi^{(N)} u)>1.$$
\end{itemize}
\end{lemma}

In particular, mode-locking is not a measure theoretic notion but a topological property of the dynamics above the support of $\mu$ (note that the second item does not depend on the choice of the metric on $\EE$).

\begin{proof}Recall that (nearly) by definition
 $\rho(x,\phi)=\lim\frac{1}{n}\sigma_n(x).$
Hence the implication $(i)\Rightarrow (ii)$. Let us show $(ii)\Rightarrow (i)$. 

Under the assumptions of $(ii)$, uniform continuity of $u\mapsto w(\phi^{(N)} u)$ gives $\epsilon>0$ and an open set $V$ of points $x\in \supp(\mu)$ such that $\sigma_N(x)>1+\epsilon$. Kingman's theorem applied to the superadditive sequence $(\sigma_n)$ gives
 $$\rho(\mu,\phi)=\sup_{n>0} \frac{1}{n}\int\sigma_n(\cdot,\phi)d\mu\geq \frac{1}{N}\int\sigma_N(\cdot,\phi)d\mu.$$
We have $\sigma_N\geq 0$ since $\phi$ is strictly increasing. From $\mu(V)>0$ we deduce $(i)$.
\end{proof}

\section{Rotation numbers of perturbations of linear cocycles and diffeomorphisms}\label{s.varyingbundles}
Consider a dominated splitting for a linear cocycle or a diffeomorphism restricted to a hyperbolic set. That dominated splitting admits a continuation when the dynamics is perturbed. Given a path of perturbations, the projectivization of a 2-dimensional bundle of the dominated splitting is then a path $\EE_t$ of circle bundles. We would like to define a rotation number along that path.
The problem is that there is no canonical identification of $\EE_t$ to $\EE_0$. Cohomology invariance will help us.

\subsection{Rotation numbers along paths of circle bundles.}\label{s.oafan}
We now allow the bundle $\EE$ to vary with $t$. 
Let  $\GG$ be a fibre bundle with base $X$ and endowed with some distance $d_\GG$.  Given two maps with same domain $f,g\colon K\to \GG$ define the distance 
$$d_{\sup{}}(f,g)=\sup_{u\in K}d(fu,gu).$$   
Let $(\cE,d)$ be a space of circle subbundles of $\GG$ endowed with the following distance, where the infimum is taken over the homeomorphisms $h\colon \EE \to \FF$ fibering on $\Id_X$:
$$d(\EE,\FF)=\inf_{h}d_{\sup{}}(\Id_\EE,h).$$
 Assume that there exists $\alpha>0$ and a family of homeomorphisms $h_{\EE,\FF}\colon \EE\to \FF$  fibering over $\Id_X$ indexed by the pairs $(\EE,\FF)\in \cE^2$ with $d(\EE,\FF)<\alpha$ such that 
\begin{itemize}
\item[($i$)] the map $(\EE,\FF,u)\mapsto h_{\EE,\FF}(u)$ is continuous on its domain,
\item[($ii$)] $d(\EE,\FF)$ and $d_{\sup{}}(\Id_\EE,h_{\EE,\FF})$ are equivalent, that is, there exists $C>1$ such that for all $\EE,\FF\in \cE$ with $d(\EE,\FF)<\alpha$ we have 
$$C^{-1}.d(\EE,\FF)<d_{\sup{}}(\Id_\EE,h_{\EE,\FF})<C.d(\EE,\FF).$$
\end{itemize}

\begin{example}\label{e.vectorbundlesexample}
Let $\GG=\PP V=\sqcup_{x\in X}\PP V_x$ be the projectivization of the Euclidean bundle $V$ and let $\cE$ be the set of continuous projective line subbundles of $\GG$. Define then $h_{\EE,\FF}$ on each fiber $\GG_x$ as the orthogonal projection on $\FF_x$ (which is well-defined on $\EE_x$ when $\EE$ is close enough to $\FF$).
\end{example}

\begin{example}\label{e.diffeosexample}
Take $\GG=X\times \PP TM$, where $M$ is a smooth Riemannian manifold and $\PP TM$ is the projectivization of the tangent bundle of $M$.  Let $\cE$ be the set of continuous circle subbundles $\EE$ of  $\GG$ such that each fiber $\EE_x$ is a projective line in a fibre $\PP T_{p_\EE(x)}M$, where $p_\EE \colon X\to M$ is a continuous map. Define the homeomorphisms $h_{\EE,\FF}$ as follows: 

If $\EE$ and $\FF$ are close enough, then the fibers $\EE_x$ and $\FF_x$ lie in two neighboring fibers $\PP T_pM$ and $\PP T_qM$. Identify them by parallel transport along the geodesic from $y$ to $z$ and project orthogonally on $\FF_x$.
\end{example}

Let $Sk^+_{\cE,T}$ be the set of pairs $(\EE,f)$ where $\EE\in \cE$ and $f$ is an element of $Sk_{\EE,T}^+$ as defined in the previous section. 
Extend the $C^0$-distance on $Sk^+_{\EE,T}$ to a distance on $Sk^+_{\cE,T}$ by taking the following infimum over the homeomorphisms $h\colon \EE\to \FF$ fibering over $\Id_X$:
$$\dist\bigl[(\EE,f), (\FF,g)\bigr]=\inf_h \bigl[d_{\sup{}}(\Id_\EE,h)+d_{\sup{}}(f,h^{-1}\circ g\circ h)\bigr].$$

Note that $\dist\bigl[(\EE,f), (\FF,g)\bigr]\geq d(\EE,\FF)$. Given a path $\phi=(\EE_t,f_t)_{t\in [0,1]}$ in $Sk^+_{\cE,T}$, covering the path $\EE_t$ by finitely many balls of radius $\alpha$ and composing maps $h_{\EE,\FF}$, one easily builds a family of homeomorphisms  fibering over $\Id_X$:
\begin{equation}
h_t \colon \EE_0\to \GG\label{e.ht}
\end{equation}
such that $(t,u)\mapsto h_tu$ is continuous.  The path
$$\phi_{0}=\bigl(h_t^{-1}\circ f_t\circ h_t\bigr)_{0\leq t\leq 1}$$ 
 in $Sk^+_{\EE_0,T}$ is not uniquely defined, but any other choice of a path $h_t$ would give another path $\psi_{0}$ that coincides with $\phi_{0}$ in cohomology. By \cref{p.conjugacyinvariant}, the rotation number 
$$\rho(x, \phi):=\rho(x, \phi_0)$$
is hence well-defined for $\mu$-a.e. points $x\in X$. 
One then extends \cref{p.rotationmorphism} to varying bundles:

\begin{proposition}\label{p.rotationmorphismvarying2}
Fix a path $\phi$ in $Sk_{\cE,T}^+$ and a measure $\mu\in \cM(T)$. The rotation number $\rho(\cdot, \phi)$ is defined $\mu$-a.e., bounded, $T$-invariant, measurable and only depends on the homotopy class of $\phi$ with fixed extremities. Moreover, 
$$\rho(\mu,\phi):= \int_X \rho(\cdot,\phi)d\mu$$
 defines a continuous map $\rho\colon \cM(T)\times P_1(Sk_{\cE,T}^+)\to \RR$. 
\end{proposition}

Again, we call $\rho(\mu,\phi)$ the {\em rotation number} of $\phi$ above the measure $\mu$. Given a compact space $K$, we say that two continuous families $(\EE_\kappa,f_\kappa)_{\kappa\in K}$ and $(\FF_\kappa,g_\kappa)_{\kappa\in K}$ of skew-products in $Sk_{\cE, T}^+$   {\em coincide in cohomology} if for all $\kappa$ the maps $f_\kappa$ and $g_\kappa$ are cohomologous, that is, if there exists a homeomorphism $h_\kappa\colon \EE_\kappa\to \FF_\kappa$ fibering over $\Id_X$ such that 
$$f_\kappa\circ h_\kappa=h_\kappa\circ g_\kappa.$$

\begin{proof}[Proof of \cref{p.rotationmorphismvarying2}]
The only points that need to be explained are the continuity and the invariance by  isotopy with fixed extremities, the rest is given by \cref{p.conjugacyinvariant}. 

Consider an isotopy $(\phi_s)_{0\leq s\leq 1}=(\EE_{t,s},f_{t,s})_{0\leq s,t\leq 1}$ with fixed extremities between two paths $\phi_0=(\EE_{t,0},f_{t,0})$ and $\phi_1=(\EE_{t,1},f_{t,1})$ of $Sk_{\cE, T}^+$.
Let $\alpha>0$ be as defined in the beginning of \cref{s.oafan}.
Cutting the isotopy into finitely many small pieces, showing invariance by isotopy reduces to showing it for each of these pieces.
Therefore we may assume that $d(\EE_{t,s},\EE_{t,0})<\alpha$ for all $s,t\in [0,1]$. Then conjugate by $h_{\EE_{t,0},\EE_{t,s}}$: the isotopy $(\phi_s)$ coincides in cohomology to an isotopy $(\EE_{t,0},\widehat{f}_{t,s})_{0\leq s,t\leq 1}$ with fixed extremities. Conjugating by the path of homeomorphisms  $h_t\in \Homeo_{\Id}(\EE_0,\EE_t)$ of \cref{e.ht}, we get that it coincides in cohomology to an isotopy $\psi_s=(\EE_{0,0},\tilde{f}_{t,s})_{0\leq s,t\leq 1}$ with fixed extremities in $Sk^+_{\EE_{0,0},T}$. As a consequence of the invariance by isotopy in  \cref{p.rotationmorphism} and of cohomology invariance (\cref{p.conjugacyinvariant}), the rotation numbers of $\phi_0$ and $\phi_1$ coincide. 

The proof of continuity goes the same way: take a sequence $\mu_n$ of measures tending to $\mu_\infty$ and a sequence $\phi_n$ of paths of $Sk_{\cE, T}^+$ tending to $\phi_\infty$ for the topology of uniform convergence. This corresponds to taking a continuous family 
$(\EE_{t,n},f_{t,n})_{(t,n)\in K}$ where $K$ is the compact set $[0,1]\times (\NN\cup \{\infty\})$. One may assume that all paths $\phi_n$ are at maximum distance $<\alpha$ from $\phi_\infty$.   As in the previous paragraph, the family $\bigl(\phi_n\bigr)_{n\in \NN\cup\{\infty\}}$  then coincides in cohomology to a continuous family  $\bigl(\psi_n\bigr)_{n\in \NN\cup\{\infty\}}$ of paths in $Sk^+_{\EE_{0,\infty},T}$. The continuity given by \cref{p.rotationmorphism} and the cohomology invariance imply that $\rho(\mu_n,\phi_n)$ tends to $\rho(\mu_\infty,\phi_\infty)$.
\end{proof}

The obvious extension of \cref{p.conjugacyinvariant} to this setting is the following:
\begin{proposition}\label{p.conjugacyinvariant2}
If two paths $(\EE_t,f_t)$ and $(\FF_t,g_t)$ in $Sk_{\cE, T}^+$ coincide in cohomology, then they have same rotation numbers.
\end{proposition}

We say that a path $(\EE_t,f_t)$ in $Sk_{\cE,T}^+$ is {\em strictly increasing} if it is cohomologous to a strictly increasing path (in our previous definition -- see \cref{s.rotmorphism.modelock}) in $Sk^+_{\EE_0,T}$. \cref{l.modelocking.increasing} translates into:

\begin{lemma}\label{r.strictincreasing.varying.case}
A skew-product $(\EE,f)\in Sk_{\cE,T}^+$ is not upper semi-locked over $\mu$ if and only if the strictly increasing paths starting at $(\EE,f)$ in $Sk_{\cE,T}^+$ have strictly positive rotation number over $\mu$.
\end{lemma}

\subsection{Perturbations of linear cocycles}

Let $V$ be a vector bundle with compact base $X$ and let  $T\colon X\to X$ be a homeomorphism. Fix a  measure $\mu$ in the set $\cM(T)$ of Borel finite $T$-invariant measures. Let $\cA\in LC_{V,T}$ be a linear cocycle that preserves orientation on a  $2$-dimensional oriented bundle $E_\cA$ of a dominated splitting. 
That dominated splitting admits a continuation (see for instance ~\cite{BDV}), in particular there is a connected neighborhood $\cU$ of $\cA$ on which the bundle $E_\cA$ extends to a unique continuous family of bundles $$E=(E_\cB)_{\cB\in \cU}$$ 
where each $\cB\in \cU$ preserves $E_\cB$ and its orientation (defined by continuation of the orientation on $E_\cA$).

 Then going back to \cref{e.vectorbundlesexample} with the projective bundle $\GG=\PP V=\sqcup_{x\in X}\PP V_x$ and applying \cref{p.rotationmorphismvarying2} to paths of the form $\phi=(\PP E_{\cA_t},\PP\cA_{t|E_{\cA_t}})$, we get that any path $(\cA_t)$ in $\cU$ has a well-defined rotation number along $E$:
$$\rho_E\bigl[\mu,(\cA_t)\bigr]:=\rho(\mu,\phi).$$
As before, this rotation number varies continuously with the measure $\mu$  and with the path $(\cA_t)$, it depends linearly on the measure $\mu$ and it is additive for concatenation.

We assume now that $\cU$ is simply connected.\footnote{Such $\cU$ exists thanks to the local simple connectedness of the set  $LC_{V,T}$ of linear cocycles.} 
Given a cocycle $\cB\in \cU$, define its {\em rotation number relative to $\cA$ along $E$} by taking any path $(\cA_t)$ in $\cU$ from $\cA$ to $\cB$ and putting
$$\rho_{\cA,E}(\mu,\cB):=\rho_E\bigl[\mu,(\cA_t)\bigr]$$
This is well-defined by the invariance by isotopy provided by \cref{p.rotationmorphismvarying2}. Moreover it is  continuous on $\cM(T)\times \cU$ and it satisfies the following: for all $\cB,\cC\in \cU$ we have
$$\rho_{\cA,E}(\mu,\cC)=\rho_{\cA,E}(\mu,\cB)+\rho_{\cB,E}(\mu,\cC).$$

\begin{remark}
As in \cref{r.germsk} this relative rotation number depends on the choice of $\cU$, but the germ of $\cB\mapsto \rho_{\cA,E}(\cdot,\cB)$ at $\cB=\cA$ does not.
\end{remark}

In the particular case of a periodic measure, the rotation number easily relates to the variation of the arguments of the eigenvalues of the first return map. Let $x\in X$ be a $T$-periodic point of period $p$. Then the first return map of the cocycle $\cB$ to the fiber ${E_{\cB,x}}$ above $x$ is 
\begin{itemize}
\item either conjugate through an orientation preserving linear map to a similarity of angle $\theta_\cB$,
\item or it has positive real eigenvalues, in which case we put $\theta_\cB=0$,
\item or it has negative real eigenvalues, in which case we put $\theta_\cB=\pi$.
\end{itemize}
We leave the following as an exercise:

\begin{scholium}\label{r.realeig}
Let $\mu_x$ be a $T$-invariant measure supported by the orbit of $x$. Then we have the following equality:
 $$\mu_x(X)\cdot (\theta_\cB-\theta_\cA)=2\pi p\cdot\rho_{\cA,E}(\mu_x,\cB) \mod 2\pi.$$ 
\end{scholium}

\begin{lemma}[The relative rotation number is a local cohomology invariant]\label{r.locconj}Namely, there is a simply connected neighborhood $\cU$ of $\cA$ and $\epsilon>0$ such that if $\cB,\cC\in \cU$ satisfy that $\PP\cB_{|E_\cB}$ and $\PP\cC_{|E_\cC}$ are   conjugate by a homeomorphism 
\[h_1\colon  \PP E_\cB\to \PP E_\cC\]
fibering over $\Id_X$ and $\epsilon$-close to $\Id_{\PP E_\cB}$, then their relative rotation numbers coincide:
$$\rho_{\cA,E}(\mu,\cB)=\rho_{\cA,E}(\mu,\cC).$$
\end{lemma}

\begin{proof}Choose a path $E_t$ of bundles between $E_0=E_\cB$ and $E_1=E_\cC$ and an $\epsilon$-small path of homeomorphisms $h_t\colon \PP E_\cB\to \PP E_t$ fibering over  $\Id_X$ going from $h_0=\Id_{\PP E_\cB}$ to $h_1$. The path $\psi=(\PP E_t,h_t\circ\PP\cB_{|E_{\cB}}\circ h_t^{-1})$ from $(\PP E_\cB,\PP\cB_{|E_\cB})$ to $(\PP E_\cC,\PP\cC_{|E_\cC})$ is cohomologically constant. Then let $\cB_t$ and $\cC_t$ be paths in $\cU$ from $\cA$ to $\cB$ and $\cC$ respectively, and define the paths 
$\phi_\cB=(\PP E_{\cB_t},\PP\cA_{t|E_{\cB_t}})$ and $\phi_\cC=(\PP E_{\cC_t},\PP\cA_{t|E_{\cC_t}})$.
Then the paths $\psi*\phi_\cB$ and $\phi_\cC$ are isotopic with fixed extremities in the locally simply connected space $\cE$ defined in  \cref{e.vectorbundlesexample}, provided that $\cU$ and $\epsilon$ are chosen small enough, therefore they have same rotation numbers. By \cref{p.conjugacyinvariant2}, the rotation number of $\psi$ is $0$, which ends the proof.
\end{proof}

\subsection{Perturbations of diffeomorphisms}
Let $f$ be a diffeomorphism on a compact manifold $M$ and let $K$ be a hyperbolic set of $f$. Fix a measure $\mu$ in the set $\cM(f_{|K})$ of finite Borel $f$-invariant measures supported in $K$. Assume moreover that $Df$ preserves orientation on a $2$-dimensional oriented bundle $E_f$ of a dominated splitting over $K$. 
Again, the hyperbolic set and the dominated splitting admit a continuation, in particular there is a connected $C^1$-neighborhood $\cU$ of $f$ such that
\begin{itemize}
\item there is a unique continuous family compact sets $(K_g)_{g\in \cU}$ and a unique continuous family of homeomorphisms $h_g\colon K_g\to K$
such that $h_f=\Id_K$ and each restriction $g_{|K_g}$ is topologically conjugate to $f_{|K}$ by $h_g$,
\item the bundle $E_f$ extends uniquely to a continuous family $E={\bigl(E_g\bigr)}_{g\in \cU}$
such that each $E_g$ is a $Dg$-invariant $2$-dimensional subbundle of $TM_{|K_g}$.
\end{itemize}
This gives then circle bundles $\PP E_g\overset{\pi}{\longrightarrow}K_g\overset{h_g}{\longrightarrow}K$ over $K$, which puts ourselves in the setting of \cref{e.diffeosexample} with $X=K$. 
Applying \cref{p.rotationmorphismvarying2} to paths of the form $\phi=(\PP E_{g_t},\PP Dg_{|E_{g_t}})$, we get that any path $(g_t)$ in $\cU$ has well-defined rotation numbers along $E$:
$$\rho_E\bigl[\mu,(g_t)\bigr]:=\rho(\mu,\phi).$$
When $\cU$ is simply connected\footnote{Such $\cU$ exists thanks to the local simple connectedness of the set of diffeomorphisms.},
define the {\em rotation number of $g\in \cU$ relative to $f$ along $E$ over $\mu$} by
$$\rho_{f,E}(\mu,g):=\rho_E\bigl[\mu,(g_t)\bigr]$$
where $(g_t)$ is any path in $\cU$ from $f$ to $g$. Again, that rotation number is continuous on $\cM(T)\times \cU$, where $\cU$ is endowed with the $C^1$-topology, and it satisfies the following rule:
for all $g,g'\in \cU$ we have
$$\rho_{f,E}(\mu,g')=\rho_{f,E}(\mu,g)+\rho_{g,E}(\mu,g').$$

\begin{remark}
Again, the germ of the relative rotation number map $g\mapsto \rho_{f,E}(\cdot,g)$ at $f=g$ does not depend on the choice of $\cU$.
\end{remark}

\section{Proofs of the main results.}\label{s.mainth}

\subsection{Proof of  \cref{t.cocyclewithsimpleLyap}}\label{s.theoremA}
For the sake of clarity, we first prove it when the map $T$ in the statement of  \cref{t.cocyclewithsimpleLyap} is itself a transitive SFT.
This case already allows to present the essential argument in a non-technical way.
Then we explain how to adapt it to the general case in the two later sections.
            
\subsubsection{Proof of  \cref{t.cocyclewithsimpleLyap} when $T$ is a transitive SFT}\label{s.ofiwheofa}
Given $x\in X$, we denote by $\Orb(x)=\{T^nx,n\in \ZZ\}$ its orbit.
Then $x$ be a $T$-periodic point.

The arguments other than the rotation number are fairly classical, we will only expose them concisely.
 A matrix in $\GL(d,\RR)$ can be perturbed so that it has {\em generic spectrum}, that is, it has at most two eigenvalues (counted with multiplicity) of each modulus: if it has only one eigenvalue of some modulus then it is real, and if two eigenvalues have the same modulus then they are non-real (and conjugate).  
Smoothly perturb the cocycle so that the first return at $x$ has generic spectrum. Assume it has $\ell\geq 1$ pairs of non-real eigenvalues. Then it is enough to prove the following:

\begin{claim}\label{c.wpeifhw}
There is a smooth perturbation of the cocycle $\cA$ and a new periodic point with generic spectrum and at most $\ell-1$ pairs of non-real eigenvalues. 
\end{claim}

\begin{proof}[Proof of \cref{c.wpeifhw}]
The set $X$ is infinite, so $\Orb(x)$ admits a homoclinic point, that is, a point $y\notin \Orb(x)$ whose $\alpha$- and $\omega$-limits are $\Orb(x)$.

Let us show that the cocycle can be smoothly perturbed so that the finest dominated splitting on $\Orb(x)$ extends to the compact set 
$$K=\Orb(x)\cup \Orb(y).$$
If the cocycle restricted to $\Orb(x)$ has a dominated splitting $E\oplus F$ of index $i$, then there are unique $i$- and $(d-i)$-planes $E_{y}$ and $F_{y}$ in $V_y$ such that on has the positive iterates $\cA^k(E_y)$ accumulate on $E$ and the negative iterates  $\cA^{-k}(E_y)$ accumulate on $F$.

 Since $y$ is an isolated point of $K$, one may obtain any small perturbation of the linear map $A(y)$ by a smooth perturbation of the linear cocycle $\cA$, without changing $\cA$ above $K\setminus \{y\}$. Replacing $\cA$ by such a perturbation, one may assume that the bundles $E_{y}$ and $F_{y}$ transverse. Define then $E_{T^ky}:=\cA^k(E_y)$ and $F_{T^ky}:=\cA^k(F_y)$ for all $k$. This extends the dominated splitting $E\oplus F$ on $\Orb(x)$ to $K$. This can be done simultaneously for all indices of domination, so that the finest dominated splitting on $\Orb(x)$ extends to $K$. By assumption on the spectrum of $\cA$ along $\Orb(x)$, that finest dominated splitting has $\ell$ bundles of dimension $2$ and all others of dimension $1$.

The shadowing lemma gives a sequence of positive real numbers $\epsilon_n\to 0$ and a sequence $x_n\in X$ of periodic points of periods $2p_n$ such that 
$$\dist(T^kx_n,T^ky)\leq \epsilon_n,$$
for all integers $-p_n< k\leq p_n$.
A dominated splitting on a compact set extends to any compact invariant set in a small neighborhood of that first set. So, for $N$ large enough, the finest dominated splitting on $K$ extends to the compact set
$$L_N=K\cup \bigcup_{n\geq N} \Orb(x_n).$$ 
The dimensions of the bundles of that splitting imply  that the first return of $\cA$ to each ${x_n}$ has at most $\ell$ pairs of eigenvalues (counted with same multiplicity) with same moduli.
Let $E$ be one of the $2$-dimensional bundles of that finest dominated splitting.
There are two cases:
\begin{itemize}
\item either the cocycle $\cA$ does not preserve any continuous orientation on the restriction ${E}_{K}$ of $E$ to $K$. But the linear cocycle $\cA$ preserves some orientation on $E_{\Orb(x)}$ since the first return linear map has a pair of complex conjugate eigenvalues, by assumption. Therefore, for large $n$, the linear cocycle $\cA$ does not preserve any orientation on the restriction of $E$ to the orbit of $x_n$ and the first return to $E_{x_n}$ has real eigenvalues.
\item or the cocycle $\cA$ preserves a continuous orientation on $E_{K}$. Then if we take the integer $N$ above large enough, $\cA$ also preserves a continuous orientation on $E:=E_{L_N}$. We apply a rotation number argument. 

Denote by  $\mu_{x}$ and $\mu_{x_n}$ the invariant probability measures supported by the orbits of $x$ and $x_n$. The sequence $\mu_{x_n}$ converges to $\mu_x$ for the weak-* topology.
Take now a small simply connected neighborhood $\cU$ of $\cA$. By \cref{r.realeig}, changing the argument of the complex eigenvalues in $E$ at $x$ gives a small path $\bigl(\cA_t\bigr)_{t\in [0,1]}$ in $\cU$ starting at $\cA_0:=\cA$ whose rotation number $\rho_{\cA,E}(\mu_x,\cA_1)$ is positive. 
The periods $p_n$ of the points $x_n$ tend to infinity and $\rho_{\cA,E}(\mu_{x_n},\cA_1)\to \rho_{\cA,E}(\mu_x,\cA_1)$ by continuity of the rotation number,
therefore  $\pi<p_n\cdot \rho_{\cA,E}(\mu_{x_n},\cA_1)$ for large $n$.
The continuity of $t\mapsto \rho_{\cA,E}(\mu_{x_n},\cA_t)$ and \cref{r.realeig} give a parameter $t$ such that the  first return to  the continuation of $E_{x_n}$ has real eigenvalues. 
\end{itemize}
In both case, a further smooth perturbation gives generic spectrum at $x_n$ with at most $\ell-1$ pairs of non-real eigenvalues, which ends the proof of \cref{c.wpeifhw}.
\end{proof}
This ends the proof of \cref{t.cocyclewithsimpleLyap} for transitive SFTs. Before we proceed to the general case, we  need a few considerations on lifting vector bundles and linear cocycles.

\subsubsection{Lifts of vector bundles and linear cocycles}\label{s.oweifw}  Assume that a continuous transformation $T\colon X\to X$ is a factor of another $\sigma \colon \Omega\to \Omega$, that is, the following diagram commutes:
\begin{align}\xymatrix{\Omega\ar[rr]^\sigma \ar[d]^\pi && \Omega\ar[d]^\pi \\
            X\ar[rr]^{T} && X}\label{d.owefiwo}\end{align}
Given a vector bundle $V$ on $X$, define $V_\Omega:=\sqcup_{\omega\in \Omega}\{\omega\}\times V_{\pi(\omega)}$.
For any local trivialization $\phi_{U'}\colon U'\times \RR^d\to V_{U'}$ of the vector bundle above an open set $U'\subset X$, let $U:=\pi^{-1}U'$ and define 
\[\begin{array}{rlcl}
\phi_U\colon& U\times \RR^d&\to& V_U:=\sqcup_{\omega\in U}\{\omega\}\times V_{\pi(\omega)} \\
&(\omega,v)&\mapsto&\Bigl(\omega,\phi_{U'}(\pi(\omega),v)\Bigr)
\end{array}\]
Those maps $\phi_U$  are trivializations of a unique vector bundle structure on $V_\Omega$. We have a canonical surjection $\widetilde{\pi}\colon (\omega,v)\in V_\Omega \to v\in V$. For any continuous linear cocycle $\cA\colon V\to V$, define
\[\begin{array}{rlcl}
\cA_\Omega\colon& V_\Omega&\to& V_\Omega. \\
&(\omega,v)&\mapsto&(\sigma \omega,\cA v)
\end{array}\]
This is a continuous linear cocycle and we have the following commuting diagram:
\[\xymatrix{V_\Omega\ar[rr]^{\cA_\Omega} \ar[d]^{\widetilde\pi} && V_\Omega\ar[d]^{\widetilde\pi} \\
            V\ar[rr]^{\cA} && V}\]
 \begin{remark}\label{r.woefieona}
 The preimage by $\widetilde\pi$ of a dominated splitting $V=E^1\oplus \ldots \oplus E^\ell$ for $\cA$ is a dominated splitting  $V=E^1_\Omega\oplus \ldots \oplus E^\ell_\Omega$ for $\cA_\Omega$. The opposite is not true, in general: a bundle of a dominated splitting for $\cA_\Omega$ does not necessarily project on a vector subbundle of $V$.
 \end{remark}

\subsubsection{Proof of  \cref{t.cocyclewithsimpleLyap} in the general case}\label{s.ofiwheofa2}
Assume that $T\colon X\to X$ is a factor of a transitive SFT $\sigma\colon \Omega\to \Omega$, as in diagram (\ref{d.owefiwo}). By \cref{r.oginoga}, we may assume that $\sigma$ is a two-sided shift. 
Let $x\in X$ be the image by $\pi$ of a $\sigma$-periodic point $\xi$. Then $x$ is itself $T$-periodic. 

Smoothly perturb the cocycle so that the first return at $x$ has generic spectrum. Assume it has $\ell\geq 1$ pairs of non-real eigenvalues. As in \cref{s.ofiwheofa} it is enough to prove the following:

\begin{claim}\label{c.wpeifhw2}
There is a smooth perturbation of the cocycle $\cA$ and a new periodic point with generic spectrum and at most $\ell-1$ pairs of non-real eigenvalues. 
\end{claim}

\begin{proof}[Proof of \cref{c.wpeifhw2}] The set $X$ is infinite so we find $b\in X\setminus \Orb(x)$ and $\beta\in \Omega$ such that $\pi(\beta)=b$.  
Since $\sigma\colon \Omega\to \Omega$ is a transitive SFT, there is a point $\gamma\in \Omega$ arbitrarily close to  $\beta$ whose $\alpha$- and $\omega$-limits are $\Orb(\xi)$. Taking $\gamma$ sufficiently close $\beta$, we have $y_0=\pi\gamma\in  X\setminus \Orb(x)$. 
Then the sequence $y_k=\pi(\sigma^k\gamma)$ satisfies:
\begin{itemize}
\item $Ty_k=y_{k+1}$ for all $k\in \NN$,
\item its $\alpha$- and $\omega$-limits are $\Orb(x)$,
\end{itemize}
In other words, $(y_k)_{k\in \ZZ}$ is a homoclinic orbit for $x$. 
Notice that while $y_k\notin \Orb(x)$ for all $k\leq 0$, one may have $y_k\in \Orb(x)$ for $k$ large enough. 
The finest dominated splitting of the restriction $\cA_{|\Orb(x)}$ has as previously all bundles of dimension $1$ or $2$. The problem here is that $T$ restricted to 
\[K:=\Orb(x)\cup \{y_k,k\in \ZZ\}\]
is not invertible and that prevents us from doing as in \cref{s.ofiwheofa}: in general no perturbation of $\cA$ extends the dominated splitting on  $V_{\Orb(x)}$ to $V_K$.

In order to deal with this, we consider the lift $\cA_\Omega$ of $\cA$ to the bundle $V_\Omega$ as defined in \cref{s.oweifw}. The base transformation $\sigma$ is now invertible.  We work close to the set 
\[\widehat{K}=\Orb(\xi)\cup \Orb(\gamma).\]
By \cref{r.woefieona}, the finest dominated splitting of the restriction $\cA_{|\Orb(x)}$ lifts by $\widetilde \pi$ to a dominated splitting for $\cA_{\Omega|\Orb(\xi)}$. That dominated splitting has again $\ell$ bundles of dimension $2$ and the rest of dimension $1$.
Denote now by $A_\Omega(\gamma)\colon {V_\Omega}_\gamma \to {V_\Omega}_{\sigma\gamma}$ the restriction of $\cA_\Omega$ to the fiber of ${V_\Omega}$ above $\gamma$.

Note  that since $y_0\notin \Orb(x)$ and the $\alpha$- and $\omega$-limits of $(y_k)_{k\in \ZZ}$ are $\Orb(x)$, the point $y_0$ is isolated in the set $K$, therefore one may obtain {\bf any} small perturbation of the linear map $A(y)\colon V_y\to V_{Ty}$ by a smooth perturbation of the cocycle $\cA$ without changing it on $K\setminus \{y\}$. One obtains then {\bf any} small perturbation of the linear map $A_\Omega(\gamma)$ by a smooth perturbation of the cocycle  $\cA$ without changing $\cA_\Omega$ on $\widehat{K}\setminus \{\gamma\}$. As in the previous proof, we find then a small perturbation of $\cA$ such that the dominated spitting of $\cA_\Omega$ on $\Orb(\xi)$ extends to $\widehat{K}$.

The shadowing lemma gives a sequence $\xi_n$ of periodic points of periods $2p_n$ and a sequence $\epsilon_n\to 0$ such that 
\begin{align}\dist(\sigma^k\xi_n,\sigma^k\gamma)\leq \epsilon_n,\label{e.oarea}
\end{align}
for all integers $-p_n< k\leq p_n$. For $N$ large enough, the finest dominated splitting on $\widehat K$ extends to the compact set
$$\widehat L_N=\widehat K\cup \bigcup_{n\geq N} \Orb(\xi_n).$$ 
Let $x_n=\pi(\xi_n)$.

\begin{remark}\label{r.aeronewe}
We have $y_0=\pi(\gamma)$ isolated in $\{y_k=\pi(\sigma^k\gamma),k\in \ZZ\}$. By \cref{e.oarea}, for all $n$ large enough, for all nonzero $-p_n< k\leq p_n$ we have $T^kx_n\neq x_n$. Thus $x_n$ has same period $2p_n$ as $\xi_n$ and $\pi$ sends the orbit of $\xi_n$ bijectively on the orbit of $x_n$.
 
In particular, the first return of $\cA$ to $x_n$ has the same spectrum as the first return of $\cA_\Omega$ to $\xi_n$, and if moreover $n\geq N$, the dimension of the dominated splitting above means that there are at most $\ell$ pairs of complex conjugate eigenvalues.
\end{remark}

Let $ E$ be one of the bundles of dimension $2$ of that finest dominated splitting.
We have then the same two cases as in \cref{s.ofiwheofa}:
\begin{itemize}
\item either the cocycle $\cA_{\Omega}$ does not preserve any continuous orientation on the restriction ${E}_{\widehat K}$. But  the  cocycle $\cA$ preserves some orientation on $E_{\Orb(x)}:=\widetilde\pi E_{\Orb(\xi)}$, thus so does the lift $\cA_{\Omega}$  on $E_{\Orb(\xi)}$. Therefore, for large $n$, the linear cocycle $\cA_\Omega$ does not preserve any orientation on the restriction of $E$ to the orbit of $\xi_n$, that is the first return of $\cA_\Omega$ to $E_{\xi_n}$ has real eigenvalues.

%
\item or the cocycle $\cA_\Omega$ preserves a continuous orientation on $E_{\widehat K}$. If we take the integer $N$ above large enough, $\cA_\Omega$ also preserves a continuous orientation on $E:=E_{\widehat L_N}$. One applies the rotation number argument. 

Denote by  $\mu_{\xi}$ and $\mu_{\xi_n}$ the invariant probability measures supported by the orbits of $\xi$ and $\xi_n$. The sequence $\mu_{\xi_n}$ converges to $\mu_\xi$ for the weak-* topology.
Take now a small smooth path $\bigl(\cA_t\bigr)_{t\in [0,1]}$ starting at $\cA_0:=\cA$ such that the arguments of the eigenvalues of the first return of $\cA_t$ to the continuation of $E_x$ change with $t$. Then the arguments of the eigenvalues of the first return of $\cA_{t,\Omega}$ to the continuation of $E_{\xi}$ change with $t$.
Then as in the proof of \cref{c.wpeifhw} we find  a parameter $t$ and $n\in \NN$ such that the first return of $\cA_{t,\Omega}$  has real eigenvalues by restriction to the continuation of $E_{\xi_n}$. 
\end{itemize}
Since there are only $\ell$ bundles of dimension $2$ in the finest dominated splitting along $\xi_n$, we got in both case a perturbation $\cB$ of $\cA$ such that  the first return  of  $\cB_\Omega$ to $\xi_n$ has at most $\ell-1$ pairs of complex conjugate eingenvalues. By \cref{r.aeronewe}, so does the first return of $\cA$ to $x_n$. One concludes by a last generic perturbation that makes the spectrum generic at $x_n$ with at most $\ell-1$ pairs of non-real eigenvalues. 
\end{proof}

This ends the proof of \cref{t.cocyclewithsimpleLyap}.

\subsection{Mode-locking and dominated splittings}\label{a.p.dominationmodelocking}
We  prove  \cref{p.dominationmodelocking}. One could easily adapt the arguments of \cite{ABD} to our setting. We propose other arguments.

We may assume that $V=E$ is an oriented $2$-dimensional vector bundle and that for each $x\in X$ the linear operator $A(x)$ has norm $1$: dominated splittings and mode-locking are preserved by multiplication of a linear cocycle $\cA$ by a continuous scalar function $X\to ]0,+\infty[$.

Consider the induced action $\PP \cA$ on the projective bundle $\EE:=\PP E$, where each fiber is endowed with the distance sine of the angle. A {\em fibered $\epsilon$-chain} is a sequence  $u_0,\ldots, u_n\in \EE$ such that for each $i$, $\PP\cA u_i$ and $u_{i+1}$ are at distance $<\epsilon$ and in the same fiber of $\EE$.  

\begin{lemma}\label{l.nondom}
For all $\epsilon>0$, there exists $\delta>0$ and $N>0$ such that if $n>N$ and 
$$\frac{1}{n}\log\|\cA^n(x)\|<\delta,$$
then for each $u\in\PP V_x$ and each $v\in\PP V_{T^nx}$ there is a fibered $\epsilon$-chain of $\PP \cA$ from $u$ to $v$.
\end{lemma}

The following is left to the reader:

\begin{lemma}\label{l.homogeneous.space}
Let $G$ be a topological group endowed with a left invariant metric. Let a sequence $g_1,\ldots, g_n$ in $G$. Let $\epsilon>0$ and fix a point $h$ in the $n\epsilon$-ball centered at $g=g_n\ldots g_1$.  
Then there exists a sequence $h_1,\ldots,h_n$ in $G$ such that each $h_i$ is at distance $\leq \epsilon$ from $g_i$ and such that  
$h=h_n \ldots h_1$.
\end{lemma}

Taking for $G$ the Lie group $\SL(d,\RR)$ endowed with the left-invariant Riemannian metric given by the inner-product $<M,M>=\tr(^tMM)$ on its lie algebra $\mathfrak{sl}(d,\RR)$, given any matrix $K\in \SO(d,\RR)$ and any diagonal matrix $A\in \SL(d,\RR)$ with positive diagonal entries $a_i$, the distance between $AK$ and $K$ is $\sqrt{\sum \ln^2(a_i)}$. See for instance~\cite{L}.

Define the following distance on $]0,\infty[^d$: 
 given two sequences $\sigma=(\sigma_1,\ldots,\sigma_d)$ and $\tau=(\tau_1,\ldots,\tau_d)$ in $]0,\infty[^d$, let 
 $$d(\sigma,\tau)=\sqrt{\sum \ln^2\left(\frac{\sigma_i}{\tau_i}\right)}.$$
Due to the fact that any matrix in $G$ decomposes as a product $LAK$, where $L,K\in \SO(d,\RR)$ and $A$ is diagonal with strictly positive entries, we have the following consequence of \cref{l.homogeneous.space}:

\begin{corollary}\label{c.singvalues}
Consider a sequence of matrices $A_1,\ldots ,A_n\in G=\SL(d,\RR)$ and let $\sigma=\sigma_i$ be the non-decreasing sequence of singular values of the product $A_n\ldots A_1$. Let  $\tau$ be a non-decreasing sequence $0< \tau_1\ \leq \ldots \leq \tau_d$. 

Then there exists $B_1,\ldots, B_n\in \SL(d,\RR)$ such that each $B_i$ is at distance $\leq \frac{d(\sigma,\tau)}{n}$ from $A_i$ and such that $\tau$ is the sequence of singular values of the product $$B_n\ldots B_1.$$\end{corollary}

%
%
%
%
%

\begin{proof}[Proof of \cref{l.nondom}]Let $\epsilon>0$ and $a>0$. Choose a linear isometry from each fiber $V_{x}$ to $\RR^2$ so that the linear maps $A(x)$ identify to elements of $G=\SL(2,\RR)$. Note that the singular values of $M\in G$ are $\|M\|^{-1}$ and $\|M\|$.
We apply twice \cref{c.singvalues} in the simpler 2-dimensional case: for all $\eta>0$ there exist $\delta>0$ and $N>0$ such that, under the assumptions of \cref{l.nondom}, on finds $0<m<n$ and a sequence $B_k$ of $\eta$-perturbations (for the Riemannian metric above) of the linear maps $A_k=A(T^{k}x)$ that satisfy the following:
 \begin{itemize}
 \item $\|B_m\ldots B_1\|=a$,
 \item $\|B_n\ldots B_1\|=1$. 
 \end{itemize}
 Taking $\eta>0$ small enough gives that the action of $B_k$ on the projective line is $\epsilon/3$-close to that of $A_kx$
 
Let $w\in \PP V_{T^{m}x}$ be the direction of the major semiaxis (it has length $a$) of the ellipse image of the unit ball of $V_x$ by the linear map $B_m\ldots B_1$. By a standard argument of linear algebra, if $a$ is large enough then there exists $\epsilon/2$-close to any $u\in \PP V_x$ a line $u'$ that is sent by $B_m\ldots B_1$ on a line $\epsilon/3$-close to $w$. Symmetrically, $\|B_n\ldots B_{m+1}\|=a$ and for any $v\in \PP V_{T^{n}x}$, there exists a vector $w'$ $\epsilon/3$-close to $w$ that is sent $\epsilon/2$ close to $v$ by $B_n\ldots B_{m+1}$. 

This straightforwardly gives the desired fibered $\epsilon$-chain.
\end{proof}

\begin{proof}[Proof of \cref{p.dominationmodelocking}]
If $\cA$ admits a dominated splitting over $\supp \mu$, then it  is clearly mode-locked over $\mu$. 
Assume that it does not admit one.

Let $\epsilon>0$. The path $\phi_\epsilon=(\PP\cA\circ \cR_t)_{|t|\leq \epsilon}$, where $\cR_t$ is the rotation of angle $t$ on each fiber, is strictly increasing.
Let $\delta$ and $N$ be as in  \cref{l.nondom}. By \cite[Theorem A]{Y} there exist $x\in \supp \mu$ and $n\geq N$ such that 
$$\frac{1}{n}\log\|A^n(x)\|<\delta.$$
\cref{l.nondom} implies that $w(\phi_\epsilon^{(n)} u)>1$ for all $u\in \PP V_x$. Hence, by \cref{l.rotationcriteria}, $\rho(\mu,\phi_\epsilon)>0$ for all $\epsilon>0$. Thus $\cA$ is not mode-locked over $\mu$.
\end{proof}

\subsection{Proofs of \cref{t.growuptheorem,t.projstructstab}}
If the two-dimensional bundles of the finest dominated splitting were smooth, then one could smoothly rotate along the those bundles. It would then be possible to prove \cref{t.growuptheorem} by adapting standard monotonicity arguments \`a la Avila-Yoccoz (see the proof of~\cite[Proposition 6]{Y}.

Unfortunately the bundles of the finest dominated splitting are in general only H\"older continuous. We rotate instead along smooth approximating bundles.
We first prove \cref{t.projstructstab}.

\begin{proof}[Proof of \cref{t.projstructstab}]
The fact that projective hyperbolicity implies projective structural stability is \cref{p.abc} and is proved in \cref{s.restofproofs}.

We prove it the other implication. Fix a cocycle $\cA$ satisfying the assumptions of \cref{t.projstructstab}. Assume that its finest dominated splitting has a $2$-dimensional bundle  $E$. One can fix an orientation on $E$ by hypothesis. Let $F$ be the direct sum of the other bundles of the dominated splitting. Note that $V=E\oplus F$, but this is not a dominated splitting. For all $\epsilon>0$, for $H=E,F$, there is a smooth/analytic bundle $H^\epsilon$ such that 
the $C^0$-distance from $H^\epsilon$ to $H$ is $<\epsilon$ (take a standard metric on the Grassmannian space). Endow $E^\epsilon$ with the continuation of the orientation on $E$.

For all $t\in \RR$,  let $R_t(x)$ be the linear map that restricts to the rotation of angle $t$ on the oriented plane $E^\epsilon(x)$, and that restricts to $\Id$ on the space $F^\epsilon(x)$. This defines a linear cocycle $\cR_t$ on $V$ fibering on $\Id_X$. 
Let $\cA_t=\cR_t\circ \cA$ and let $E_t$ be the continuation of the bundle $E=E_0$ for $\cA_t$, which is defined for all $t$ small enough.

\begin{claim}\label{c.incresingpath}
For $\epsilon,\eta>0$ small enough, the path of skew-products $\phi=\bigl(\PP E_t,\PP \cA_t\bigr)_{|t|\leq \eta}$ in $Sk^+_{\cE,T}$, where $\cE$ is as in \cref{e.vectorbundlesexample}, is defined and strictly increasing.
\end{claim}

\begin{proof}
Let $\pi\colon V\to V$ the linear cocycle that restricts on each $V_x$ to the projection of kernel $F_x$ on $E_x$. 
Fix $\epsilon>0$  so that for any vector $u$ in some neighborhood $\cU$ of $\PP E$ in $\PP V$, the path $\pi (\cR_tu)$ is a well-defined strictly increasing path in $\PP E$ with respect to the given orientation. Let $\eta>0$ such that for all $|t|<\eta$, the continuation $E_t$ of $E$ for $\cA_t$ is defined and lies in $\cU$. For all $|t|<\eta$, the map $\cB_t=\pi_{|E_t}\circ \cA_t\circ {(\pi_{|E_t})}^{-1}$ is a linear cocycle on $E$ and a trivial computation gives that the path $\bigl(\PP\cB_t\bigr)_{|t|\leq \eta}$ is a strictly increasing path in $Sk_{\PP E,T}^+$. This means by definition that $\phi$ is strictly increasing  in $Sk^+_{\cE,T}$.
\end{proof}

Let us assume now by contradiction that  $\PP\cA$ is structurally $C^r$-stable. 
Fix $\epsilon,\eta>0$ so that the conclusion of the claim is satisfied.  Domination implies that the projective bundle $\PP E$ is normally hyperbolic for $\PP\cA$. In particular, there is no other $\PP\cA$-invariant topological circle subbundles of $\PP V$ close to $\PP E$. This implies that a homeomorphism $h_t\colon \PP V\to \PP V$ fibering over $\Id_X$, close enough to $\Id_{\PP V}$ and conjugating $\PP\cA$ and $\PP\cA_t$ has to send $\PP E_t$ on $\PP E$. In other words, the restriction $\PP\cA_{|E}$ is conjugate to $\PP\cA_{t|E_t}$ by a fibered homeomorphism close to $\Id_E$. Let $\mu$ be an invariant measure of the base of $V$ with total support.
By \cref{r.locconj},  we have $\rho_{\cA,E}(\mu,\cA_{t})=0$ for $t$ close enough to $0$, which implies with \cref{c.incresingpath,r.strictincreasing.varying.case,p.dominationmodelocking} that there is a dominated splitting on $E$. This contradicts the fact that $E$ is a bundle of the finest dominated splitting, and ends the proof of \cref{t.projstructstab}.
\end{proof}

For the sake of clarity, we first prove the following particular case of \cref{t.growuptheorem} (which already implies \cref{t.complexdominationdichotconje} straightforwardly):

\begin{proposition}\label{t.babytheorem} Let $r\geq 0$. Let $\cA=(T,A)$ be a linear cocycle such that periodic points are dense for $T\colon X\to X$. Let $E$ be an orientable $2$-dimensional bundle of the finest dominated splitting on which orientation is preserved. 

Then in any $C^r$-neighborhood of $\cA$ there is a cocycle that has an elliptic point by restriction to the continuation of $E$.
\end{proposition}

\begin{proof}[Proof of \cref{t.babytheorem}]
We start as in the proof of \cref{t.projstructstab} until \cref{c.incresingpath}.
We take $\epsilon,\eta>0$ so that the conclusion of \cref{c.incresingpath} is satisfied.
The periodic points are dense in $X$, therefore there is a sequence of positive periodic measures $\mu_n$ such that $\mu=\sum \mu_n$ is a finite measure of total support in $X$. 
By \cref{p.dominationmodelocking,r.strictincreasing.varying.case}, one finds $\hat{t}$ arbitrarily close to zero such that the relative rotation number satisfies $\rho_{\cA,E}(\mu,\cA_{\hat{t}})>0$.

By continuous dependence of the rotation number to the measure, there exists an integer $N>0$ such that, defining $\nu_N=\sum_{n=0}^{N}\mu_n$, we have  $\rho_{\cA,E}(\nu_N,\cA_{\hat{t}})> 0$. If none of the periodic points corresponding to measures $\nu_n$ had complex eigenvalues in $E_t$, from $t=0$ to $t=\hat{t}$, then by  \cref{r.realeig} the relative rotation number 
$$\rho_{\cA,E}(\nu_N,\cA_{t})=\sum_{n=0}^{N}\rho_{\cA,E}(\mu_n,\cA_{t})$$ would remain constant from $t=0$ to $t=\hat{t}$ for all $N\in \NN$, which is absurd since $\rho_{\cA,E}(\nu_N,\cA_0)=0$ ! Therefore, there exists a smooth perturbation that creates in the continuation of $E$ a pair of complex (non-real) eigenvalues along some periodic point $x\in X$. This ends the proof of  \cref{t.babytheorem}.
 \end{proof}
 
 We can now tackle the general theorem.

\begin{proof}[Proof of \cref{t.growuptheorem}] The difficulty is that a perturbation that creates an elliptic point within the first bundle may create dominated splittings within other bundles: the naive idea of doing a sequence of $\ell$ perturbations does not work. 

Since a small perturbation will not destroy an already existing elliptic point, we only have to work on the bundles that do not contain an elliptic point. We may assume that the bundles $E_1,\ldots, E_\ell$ contain no elliptic points.
Consider a sequence of positive periodic measures $\mu_n$ such that $\mu=\sum \mu_n$ is a finite measure of total support in $X$ and such that $\mu_n(X)\in \QQ$ for each $n$.
By \cref{p.dominationmodelocking}, $\PP\cA_{|E_i}$ is not mode-locked over $\mu$. Choose an orientation on each $E_i$ so that $\PP\cA_{|E_i}$ is not upper semi-locked.

Let $F$ be the sum of the bundles of the finest dominated splitting, other than the bundles $E_i$. Define $F^\epsilon$ and oriented bundles $E_i^\epsilon$ as in the proof of  \cref{t.babytheorem}.

Given $\underline{t}=(t_1,\ldots,t_\ell)\in \RR^\ell$, let $R_{\underline{t}}(x)\in GL(d,\RR)$ be the linear map that restricts to the rotation of angle $t_i$ on the oriented plane $E^\epsilon_i(x)$, and that restricts to $\Id$ on the space $F^\epsilon(x)$. This defines a linear cocycle $\cR_{\underline{t}}$ fibering on $\Id_X$. As in the proof of \cref{t.projstructstab}, let $\cA_{\underline{t}}=\cR_{\underline{t}}\circ \cA$. One finds a simply connected neighborhood $\cU$ of $\cA$ on which the finest dominated splitting of $\cA$ admits a continuation. Let $\cO$ be a neighborhood of $\underline{0}=(0,\dots, 0)$ such that $\cA_t\in \cU$ and let $E_{\underline{t}}=(E_{1,\underline{t}},\ldots, E_{\ell,\underline{t}})$ be the continuation of $E_{\underline{0}}:=(E_1,\ldots,E_\ell)$ for $\cA_t$.

One can define a rotation number of $\cA_{\underline{t}}$ relative to $\cA$ along each $E_i$, for each $\underline{t}\in \cO$.
Let  $\rho_{\cA,\underline{E}}(\mu,\cB)$ be the compound rotation number, that is, the $\ell$-uple of rotation numbers $\rho_{\cA,E_i}(\mu,\cB)$. Define then
\[\Theta\colon \underline{t}\in
\cO \mapsto \rho_{\cA,\underline{E}}(\mu,\cA_{\underline{t}}).
\]

\begin{claim}\label{c.3.5}
For any neighborhood $U \subset \cO$ of $0$, there exists an open set $0\in D\subset U$ and $z\in \RR^\ell$ such that the Brouwer degree $d_B(\Theta,D,z)$ is defined and non-zero.
\end{claim}

We first show how this claim implies the theorem. Consider the compound rotation numbers over the measures \[\nu_N=\mu_0+\ldots +\mu_N,\] 
that is, the maps \[^N\Theta(\underline{t}):=\rho_{\cA,\underline{E}}(\nu_N,\cA_{\underline{t}}).\] By continuity of the rotation number, the maps $^N\Theta$ converge uniformly to $\Theta$ on compact sets, therefore for $N$ large enough, the Brouwer degree $d_B(^N\Theta,D,z)$ is also defined and non-zero.
In particular $z$ is in the interior of $^N\Theta(D)$.

Let $\underline{t}\in D$ such that $^N\Theta(\underline{t})$ has irrational coordinates. Let $1\leq i\leq \ell$. By assumption, the restriction of $\cA$ to $E_i$ admits no elliptic point and $\mu_k(X)\in \QQ$ for $k=1,\ldots, N$. This implies by \cref{r.realeig} that $\cA_{\underline{t}}$ admits an elliptic point by restriction to the continuation $E_{i,\underline{t}}$. Hence \cref{c.3.5} implies indeed \cref{t.growuptheorem}.
\medskip

\begin{proof}[Proof of  \cref{c.3.5}]
Let $\pi_i\colon V\to V$ be the linear cocycle whose restriction on each fiber $V_x$ is the projection of direction $F_x\oplus \bigoplus_{i\neq j} E_{j,x}$ and of base $E_{i,x}$. Choosing  $\epsilon>0$ small enough, one has:
\begin{itemize}
\item for any vector $u$ in some neighborhood $\cU_i$ of $\PP {E_i}$ in $\PP V$, 
\item for any $C^1$-path $\bigl(\underline{t}_s\bigr)_{s\in \RR}$ tangent to the cone 
\[C_{i,3}=\{(t_1,\ldots, t_\ell)\in \RR^\ell:\forall j\neq i, |t_j|<3t_i\},\footnote{in particular, $t_i$ is strictly positive here.},\]
\end{itemize}
the path $\pi_i (\cR_{\underline{t}_s}u)$ is a well-defined strictly increasing path in $\PP E_i$. Choose $\eta>0$ small enough such that 
$\tB_\eta=[-\eta,\eta]^\ell\subset U$ and for all $\underline{t}\in \tB_\eta$ the continuation $E_{i,\underline{t}}$ lies in $\cU_i$.
Then the exact same way as in the proof of \cref{c.incresingpath} of \cref{t.projstructstab}, for any $C^1$-path $(\underline{t}_s)$ in $\tB_\eta$ tangent to $C_{i,3}$,  we can show that the path $\phi=(\PP E_{i,\underline{t}_s},\PP \cA_{\underline{t}_s})$ is strictly increasing.
Denote the pair of $i$-th opposite faces of the cube $\tB_\eta$ by 
\begin{align*}
\tB_{i}^-&=[-\eta,\eta]^{i-1}\times \{-\eta\}\times [-\eta,\eta]^{\ell-i}\\
\tB_{i}^+&=[-\eta,\eta]^{i-1}\times \{\eta\}\times [-\eta,\eta]^{\ell-i}.
\end{align*}
Any point of $\tB_{i}^+$ can be joined from $\underline{0}$ by a $C^1$-path $(\underline{t}_s)$ in $\tB_\eta$ tangent to $C_{i,3}$. And $0$ can be joined from any point of  $\tB_{i}^-$ by a $C^1$-path $(\underline{t}_s)$ in $\tB_\eta$ tangent to $C_{i,3}$. We chose the orientations on the bundles $E_i$ so that $\cA_{\underline{0}}=\cA$ is not upper semi-locked by restriction to $E_i$. By  \cref{r.strictincreasing.varying.case}, we get that 
\[\rho_{\cA,{E}_i}(\mu,\cA_{\underline{t}})>0 \qquad\mbox{ on  $\tB_{i}^+$}.\]
 \[\rho_{\cA,{E}_i}(\mu,\cA_{\underline{t}})\leq0  \qquad\mbox{ on  $\tB_{i}^-$}.\] Let $\alpha>0$ such that for each $i$ we have $\rho_{\cA,{E}_i}(\mu,\cA_{\underline{t}})>\alpha$ on $\tB_{i}^+$. Let $\Phi$ be the homothecy that sends the cube $\tB_\eta$ on the cube $[0,\alpha]^\ell$. The barycentric isotopy $\Phi_s$ between 
\[\Phi_0=\Theta:=(\underline{t}\in
\cO \mapsto \rho_{\cA,\underline{E}}(\mu,\cA_{\underline{t}}))\] and $\Phi_1=\Phi$ satisfies that for all $s\in [0,1]$, the image of the boundary $\Phi_s(\partial \tB_\eta)$ does not intersect the interior $]0,\alpha[^\ell$. 
Let $z=(\alpha/2,\ldots,\alpha/2)$ and let $D$ be the interior of $\tB_\eta$. Then the Brouwer degree $d_B(\Theta,D,z)$ is equal to $d_B(\Phi,D,z)$ which is non-zero.
 \end{proof}
This ends the proof of \cref{t.growuptheorem}.
\end{proof}

\section{The rest of the proofs.}\label{s.restofproofs}
This section is independent from \cref{s.rotmorphism.modelock,s.varyingbundles,s.mainth}, it gathers the proofs of the other results stated in \cref{s.structstab}. There are no particularly new or unexpected arguments here, but we wrote it in the details for the sake of completeness. 

\subsection{Proof of \cref{p.abc}}
\begin{proof}[Projective structural stability implies robust absence of elliptic points] The fact that $(b)\Rightarrow (c)$ is quite clear. Indeed, if $\cA$ can be $C^r$-perturbed into a cocycle $\cB$ that has an elliptic point, then $\PP\cB$ has a normally hyperbolic circle on which the first return map is conjugate to a rotation. A further $C^r$-perturbation changes the angle of the elliptic point, hence the rotation number along that normally hyperbolic circle by \cref{r.realeig}. Therefore $\PP\cA$ is not structurally $C^r$-stable by \cref{r.locconj}. 
\end{proof}

\begin{proof}[Projective hyperbolicity implies projective structural stability] 
We now show $(a)\Rightarrow (b)$.  Let $\cA$ be a cocycle that admits a dominated splitting $V=E_1\oplus \ldots \oplus  E_d$ into $1$-dimensional bundles. Let $\cA_k$ be a sequence of cocycles tending to $\cA$ for the topology of uniform $C^0$-convergence. We need to find a sequence of homeomorphisms $h_k$ of $\PP V$ fibering over $\Id_X$ and tending to  $\Id_{\PP V}$, such that any $k\in \NN$ large enough  conjugates $\PP\cA$ to $\PP\cA_k$.

There is $N\in \NN$ such that:
\begin{itemize}
\item  for each $k \geq N$, the cocycle $\cA_k$ admits a dominated splitting  $V=E^k_1\oplus \ldots \oplus  E^k_d$ and each sequence of bundles $(E^k_i)_{k\geq N}$ converges to $E_i$ for the $C^0$-topology,
\item  for each $k \geq N$, there is an invertible cocycle $\cB_k\colon V\to V$ fibering on $\Id_X$ such that $\cB_k$ sends $E_i$ on $E_i^k$,
\item the sequence $\cB_k$ of cocycles tends to $\Id_{V}$.
\end{itemize}
Reindexing by $k:=k+N$ and doing the substitution $\cA_k:=\cB_k\circ \cA_k\circ \cB_k^{-1}$, we may assume without loss of generality that $V=E_1\oplus \ldots \oplus E_d$ is a dominated splitting for all linear cocycles $\cA_k$, $k\geq 0$.

We now prove our result by induction on $d$. It is clearly true for $d=1$. Take $d>1$ and assume the result is true in dimension $d-1$. Let $W=E_2\oplus \ldots \oplus E_d$. Then there is a sequence $h_k\colon \PP W \to \PP W$ of homeomorphisms fibering on $\Id_X$ tending to $\Id_{\PP W}$ such that, for any integer $k$ greater than some $N\in \NN$, we have
$$\PP\cA_{|\PP W}= h_k^{-1}\circ \PP \cA_k\circ h_k.$$
With \cite{G1}, we may assume that 
\begin{itemize}
\item the metric $\|.\|$ on $V$ is {\em adapted} to the dominated splitting for $\cA$, that is, for any point $x\in X$, for any unit vectors $u\in E_{1,x}$ and $v\in W_x$, we have  $\|\cA u\|<\|\cA v\|$,
\item $E_1$ is orthogonal to $W$.
\end{itemize} 
Given $y\in \PP W_x$ and $\theta\in [0,\pi/2]$, denote by $[\theta:y]$ the set of points $z\in \PP V_x$ such that
\begin{itemize}
\item the minimum angle between the line $\RR.z$ and $W$ is equal to $\theta$, 
\item if $\theta\neq \pi/2$ then the orthogonal projection on $W$ sends $z$ on $y$.
\end{itemize}
Note that $[\theta: y]$ has cardinality $1$ if $\theta\in \{0,\pi/2\}$ and cardinality $2$ otherwise. Since they preserve the bundles $W$ and $E_1$, the actions of $\PP \cA$ and $\PP \cA_k$ factor  into actions on the set $\{[\theta:y], y\in \PP W, \theta\in [0,\pi/2]\}$, indeed $\PP\cA[\theta:y]=[\xi_y\theta:\PP\cA y]$, where each 
$$\xi_y\colon [0,\pi/2] \to [0,\pi/2]$$ is an orientation preserving diffeomorphism. 

The fact that the metric is adapted to the dominated splitting $E_1\oplus W$ implies that $\PP\cA [1:y]=[a(y):\PP\cA y]$, where $a$ is a continuous function with $0<a<1$. On the other hand,  $\PP\cA_k [1:y]=[a_k(y):\PP\cA_k y]$, where $a_k$ is a sequence of continuous functions converging uniformly to $a$.

Define a conefield $C\subset \PP V$ as the following union
$$C=\bigcup_{0\leq \theta\leq 1\atop y\in \PP W}[\theta: y].$$ This is a (strictly invariant) center-unstable cone-field for $\PP\cA$, that is,  $\PP\cA C$ is in the interior of $C$. 
Since $\cA_k$ converges to $\cA$, for any $k$ larger than some $n\in \NN$, the set $\PP\cA_k C$ is also in the interior of $C$.
In that case, the sets $C\setminus \PP\cA C$ and $C\setminus \PP\cA_k C$ are respective fundamental domains of $\cA$ and $\cA_k$ by restriction to the set $\PP V \setminus (\PP W\cup \PP E_1)$.

Let $\eta_{k,y}$ be the affine orientation preserving homeomorphism from the segment $[a(y),1]$ on the segment $[a_k(y),1]$. There exists $\epsilon>0$ and $N\geq n$ such that if $k\geq N$, then for any subset $[\theta: y]\subset  C\setminus \PP\cA C$ and any $v\in [\theta: y]$, there exists a unique $w\in [\eta_{k,y}(\theta):h_k(y)]$ such that $\dist(v,w)<\epsilon$. This defines a homeomorphism $g_k\colon C\setminus \PP\cA C \to C\setminus \PP\cA_k C$. The sequence $(g_k)_{k\geq N}$ tends to $\Id_{C\setminus \PP\cA C}$ as $k\to \infty$.

For $k\geq N$, the set $\PP V \setminus (\PP W\cup \PP E_1)$ is the disjoint union of the $\PP\cA$-iterates of $C\setminus \PP\cA C$  (resp. of the $\PP\cA_k$-iterates of  $C\setminus \PP\cA_k C$). Hence, we can extend $g_k$ into a bijection of $\PP V \setminus (\PP W\cup \PP E_1)$ by defining $g_k:=\PP\cA^n_k\circ g_k\circ \PP\cA^{-n}$ by restriction to $\PP\cA^n(C\setminus \PP\cA C)$, for each $n\in \ZZ$. We leave it to the reader to check that 
\begin{itemize}
\item $g_k$ is continuous,
\item it extends by $\Id_{\PP E_1}$ on $\PP E_1$ and by $h_k$ on $\PP W$ into a homeomorphism $\tilde{g}_k$ of $\PP V$,
\item $\tilde{g}_k$ tends to $\Id_{\PP V}$ as $k\to \infty$.
\item $\PP\cA=\tilde{g}_k^{-1}\circ \PP\cA_k\circ \tilde{g}_k$.
\end{itemize}
We just proved that $\PP\cA$ is structurally $C^0$-stable.
\end{proof}

\subsection{The converse implications}\label{s.abcequivalence}
We show the  $C^0$-regularity cases of conjectures \ref{c.complexdominationdichotconje} and \ref{c.ba}. All arguments are standard and therefore only briefly presented.

In this section, $V$ is a bundle with compact base $X$ and endowed with some continuous Euclidean metric $\|.\|$. We say that a finite sequence of matrices $A_1,\ldots, A_p$ of  $\GL(d,\RR)$ is {\em $N$-dominated} if the induced linear cocycle $\cA$ on the trivial bundle $V=\ZZ/p\ZZ\times \RR^d$ over the shift 
\[k\in \ZZ/p\ZZ\mapsto k+1\]
 admits a dominated splitting 
$V=E\oplus F$ such that for all $x\in X$, for all unit vectors $u,v\in E_x,F_x$ above $x$, we have 
$$ \|\cA^N(u)\|<\frac{1}{2} \|\cA^N(v)\|.$$
We obtain the following as a direct consequence of the main result of \cite{BGV}: 
\begin{proposition}\label{l.BGV}
Let $K$ be a compact set in $\GL(d,\RR)$ and $\epsilon>0$. There exists an integer $N\in \NN$ and  such that, if a sequence $A_1,\ldots, A_p$ in $K$ has no $N$-domination of any index, then there is an $\epsilon$-perturbation $B_1,\ldots, B_p$ of the sequence $A_1,\ldots, A_p$ such that the product $B_p\ldots B_1$ has all eigenvalues with same modulus.
\end{proposition}
Note that in \cite{BGV}, it is required that the length $p$ of the sequence be large enough. This hypothesis is indeed needed to obtain real eigenvalues after perturbation.  \cref{l.BGV} is however a straightforward consequence of the two following lemmas:

\begin{lemma}[Bonatti-G.-Vivier \cite{BGV}]
Let  $K$ be a compact set in $\GL(d,\RR)$ and $\epsilon>0$. There exists  two integers $N\in \NN$ and $P\in \NN$ such that, for any $p\geq P$, if a sequence $A_1,\ldots, A_p$ in $K$ has no $N$-domination of any index, then there is an $\epsilon$-perturbation $B_1,\ldots, B_p$ of the sequence $A_1,\ldots, A_p$ such that the product $B_p\ldots B_1$ has real eigenvalues with same modulus.
\end{lemma}

\begin{lemma}\label{l.boundedperiods}
Let $K$ be a compact set in $\GL(d,\RR)$, let $\epsilon>0$ and let $P\in \NN$. There exists an integer $N\in \NN$ and  such that, for any $p\leq P$, if a sequence $A_1,\ldots, A_p$ in $K$ has no $N$-domination of any index, then there is  an $\epsilon$-perturbation $B_1,\ldots, B_p$ of the sequence $A_1,\ldots, A_p$ such that the product $B_p\ldots B_1$ has all eigenvalues with same modulus.
\end{lemma}
This has probably already been proved many times in the litterature, but we could not find a reference. We do it here for completeness.

\begin{proof}[Idea of the proof of \cref{l.boundedperiods}]
Let $\cA$ be the linear cocycle induced by the sequence  $A_1,\ldots, A_p$ on $V=\ZZ/p\ZZ\times \RR^d$. Denote by $A$ the first return of the cocycle on $\{0\}\times \RR^d$:
$$A=A_n\ldots A_1.$$ Let $\RR^d=E\oplus F$ be a dominated splitting for $\cA$, then $A_{|E}$ is conjugate through an isometry of $E\otimes \CC$ to $\Delta+N$ where $\Delta$ is diagonal and $N$ is nilpotent. Denote by $|\lambda_{E,max}|$ the maximum modulus of the eigenvalues of $A_{|E}$. Denoting by $\|.\|$ the operator norm of a linear map, one has then 
$$A_{|E}^n=\sum \left(n \atop k\right)\Delta^{n-k}N^k$$
\begin{align*}
\|A_{|E}^n\|&\leq \sum_{k=0}^d \left(n \atop k\right)|\lambda_{E,max}|^{n-k}\|N^k\|\\
&\leq  |{\lambda_{E,max}}|^n \sum_{k=0}^d \left(n \atop k\right)\frac{\|N^k\|}{|\lambda_E|^{k}}\\
&\leq  |{\lambda_{E,max}}|^n C
\end{align*}
where $C$ is a constant that depends only on $d$, $K$ and $P$.
Likewise we show that the minimum norm $\mathfrak{m}(A_{|F}^n)=\|A_{|F}^{-1}\|^{-1}$ of $A_{|F}$ satisfies 
$$\mathfrak{m}(A_{|F}^n)\geq |\lambda_{F,min}|^nC'$$
for some other constant $C'$ that depends only on $d$, $K$ and $P$.
Thus, the ratio $|\lambda_{F,min}|/|\lambda_{E,max}|$ gives a strength of domination (that is, an $N$ for which the sequence $A_1,\ldots, A_p$ is $N$-dominated) that depends only on  $d$, $K$ and $P$.

So we proved that for every $\eta>0$ there exists an integer $N\in \NN$ such that  for any $p\leq P$, if a sequence $A_1,\ldots, A_p$ in the compact set $K$ has no $N$-domination of any index, then the eigenvalues $\lambda_1,\ldots, \lambda_d$ of the product $A$ counted with multiplicity and ordered by increasing moduli satisfy $|\lambda_d|/|\lambda_1|<1+\eta$.

Choose an orthogonal basis of $\{0\}\times \RR^d$ that makes the matrix $A$ a block upper triangular matrix and define in this basis
$$B_1:=A_1\times \left(\begin{array}{cccc}
\left|\frac{\lambda_d}{\lambda_1}\right|&&&\\
&\left|\frac{\lambda_d}{\lambda_2}\right|&&\\
&&\ddots &\\
&&& \left|\frac{\lambda_d}{\lambda_d}\right|
\end{array}\right)
$$
and $B_i=A_i$ for $i\geq 2$.
For $\eta>0$ small enough, this is an $\epsilon$-perturbation of the sequence $A_1,\ldots, A_n$. It satisfies the conclusion of the lemma.
\end{proof}

Given a linear map $A$ between two Euclidean vector spaces of dimension $d$, its {\em singular values}
$$\sigma_1(A)\leq \ldots \leq \sigma_d(A)$$
are the eigenvalues of the positive semidefinite operator $\sqrt{A^tA}$, repeated with multiplicity. They equal the semiaxes of the ellipsoid obtained as the image of
the unit ball by $A$.

\begin{lemma}[See {\cite[Theorem A]{BG1}}]\label{t.BochiGour}
A linear cocycle $\cA$ on $V$ admits a dominated splitting of index $i$ if and only if there exists $C>0$ and $\tau <1$ such that for every $x\in X$ and $n\in \NN$
$$\frac{\sigma_{i}(A^{n}(x))}{\sigma_{i+1}(A^{n}(x))}<C\tau^n,$$
where $A^n(x)=A(T^{n-1}x)\ldots A(x)$ is the isomorphism of vector spaces obtained by restriction of $\cA^n$ to $V_{x}$. 
\end{lemma}

\begin{proof}[Proof of the $r=0$ case of \cref{c.complexdominationdichotconje}] In view of \cref{p.abc}, it is enough to prove that if there is no domination of index $i$, then an arbitrarily small perturbation creates an $i$-elliptic point. 

Let  $T\colon X\to X$ be a transitive subshift of finite type. Assume that it  has no domination of index $i$. Then there is a dominated splitting $V=E\oplus F\oplus G$ such that 
\begin{itemize}
\item $E$ has dimension $<i$ and $G$ has dimension $<d-i$,
\item the restricted cocycle $\cA_{|F}$ has no domination of any index.
\end{itemize}
Here $E$ and $G$ may possibly be trivial bundles.
Since the periodic points are dense in $X$, for each $N\in \NN$ there exists a periodic point $x$ such that the cocycle restricted to the restricted bundle $F_{|\Orb(x)}$ has no $N$-domination of any index. One then applies \cref{l.BGV} to find an arbitrarily small perturbation of that cocycle and another arbitrarily small perturbation to change some lyapunov exponents within $F_{|\Orb(x)}$  (if $F$ has dimension $>2$) so that, extending the perturbation of $\cA_{|F}$ thus obtained into a $C^0$-perturbation $\cB$ of $\cA$, we have a dominated splitting $TM_{|\Orb x}=E'\oplus F'\oplus G'$ for $\cB$ where
\begin{itemize}
\item $E'$ has dimension $i-1$ and $G'$ has dimension $d-i-1$,
\item the first return map $\cB^p_{|F'}$ restricted to $F'$ has a pair of eigenvalues of same modulus. 
\end{itemize}
If $\cB^p_{|F'}$ preserves orientation, then a further perturbation gives complex eigenvalues and thus makes $x$ an $i$-elliptic point. Extend that perturbation of $\cA_{|F'}$ to a $C^0$-perturbation of $\cA$ and we are done. Unfortunately $\cB^p_{|F'}$ may not preserve orientation. The properties of a transitive subshift of finite type give  a homoclinic point for $x$, that is, a point $y$ whose $\alpha$- and $\omega$-limits are $\Orb(x)$. 
By the same reasoning as in the proof of \cref{t.cocyclewithsimpleLyap}, one finds a further perturbation that extends the dominated splitting  $E'\oplus F'\oplus G'$ to the closure $K=\overline{\Orb(y)}$ of the corresponding homoclinic orbit. Then using the shadowing property of $T$ one finds a sequence of periodic points $z_n$ that shadow twice the orbit of $y$ and whose orbits tend to $K$ for the Hausdorff topology. For $n$ large enough the dominated splitting $E'\oplus F'\oplus G'$ extends above the orbit of $z_n$, and the first return map restricted to $F'_{z_n}$ is orientable. Moreover, for all $N\in \NN$, there exists $n$ such that the cocycle restricted to the bundle $F'_{|\Orb(z_n)}$ admits no $N$-domination. Apply then again \cref{l.BGV} so that through a small perturbation the first return map on $F'_{z_n}$ has a pair of eigenvalues with same moduli. Then a further perturbation gives a pair of complex conjugate eigenvalues and $z_n$ becomes an $i$-elliptic point.
\end{proof}

\begin{proof}[Short proof of the $r=0$ case of \cref{c.ba}]
Consider a linear cocycle  $\cA\colon V \to V$ fibering on a homeomorphism $T\colon X\to X$. Let us denote by $\Sigma$ the closure of the set of periodic points of $T$.

Assume that $\cA$ has no dominated splitting of index $i$. Let us first assume that this lack of dominated splitting is visible by restriction to $\Sigma$, then 
as in the previous proof, one finds arbitrarily small $C^0$-perturbation $\cB$ of $\cA$ such that there is a $\cB$-invariant splitting $TM_{|\Orb x}=E'\oplus F'\oplus G'$ where
\begin{itemize}
\item $E'$ has dimension $i-1$ and $G'$ has dimension $d-i-1$,
\item  $E'\prec F'\prec G'$,
\item the first return map $\cB^p_{|F}$ restricted to $F'$ has a pair eigenvalues of same modulus. 
\end{itemize}

Then $\PP F'$ is a periodic normally hyperbolic finite union of circles in $\PP TM_{|\Orb P}$ for the projective map $\PP \cB$. Therefore if the cocycle $\PP\cB$ were structurally stable, then a topological conjugacy $h$ to a neighboring $\PP\cC$ and fibering over $\Id_X$ and $C^0$-close to $\Id_{\PP V}$ would have to send the normally hyperbolic circles $\PP F'$ on the normally hyperbolic circles $\PP F''$, where $F''$ is the continuation of $F'$ for the perturbation $\cC$.
On the other hand, the first return projective map $\PP \cB^p$ is either conjugate to a rotation or to a symmetry of each of the circles. In the first case, a $C^0$-perturbation of $\cB$ changes the rotation number. In the second case, a $C^0$-perturbation transforms it into a North-South dynamics of the circles. Thus $\PP\cB$ is not structurally stable. Therefore $\PP\cA$ is not either.

We are now left to consider the case where the set $\Delta$ of periodic points has a dominated splitting of index $i$.
The contrapositive of  \cref{t.BochiGour} gives a sequence $y_k\in X\setminus \Delta$ and a sequence of integers $n_k\to +\infty$ such that 
\begin{equation}
\frac{1}{n_k}\log\frac{\sigma_{i}(A^{n_k}(y_k))}{\sigma_{i+1}(A^{n_k}(y_k))}\to 0.\label{e.inegalite}
\end{equation}
The domination of index $i$ by restriction to $\Delta$ means that we may assume that the points $y_k$ are aperiodic. 
By \cref{c.singvalues}, for every $\epsilon>0$ and $N\in \NN$, there exist an aperiodic point $x\in X$, an integer $n\geq N$, and a sequence of linear maps $B_0,\ldots, B_{n-1}$ that are $\epsilon$-perturbations of the linear maps $A_k=A(T^kx)$, and such that  $$\sigma_{i-1}(A^n(x))<\sigma_{i}(B^n(x))=\sigma_{i+1}(B^n(x))<\sigma_{i+2}(A^n(x)),$$
  $$\sigma_{i-1}(A_{n-1}\ldots A_0)<\sigma_{i}(B_{n-1}\ldots B_0)=\sigma_{i+1}(B_{n-1}\ldots B_0)<\sigma_{i+2}(A_{n-1}\ldots A_0),$$
for each $0<i<d$, where $\sigma_0:=0$ and $\sigma_{d+1}=+\infty$.
There is a unique  $2$-plane $E_0\in V_{x}$ such that the singular values of the restriction ${B_{n-1}\ldots B_0}_{|E_0}$ are $\sigma_{i}(B_{n-1}\ldots B_0)=\sigma_{i+1}(B_{n-1}\ldots B_0)$. 
Let $E_k\subset V_{T^kx}$ be the image of $E_0$ by the linear map $B_{k-1}\ldots B_0$. Let $F_0\subset V_{x}$ be the orthogonal of $E_0$ and let $F_k\subset V_{T^kx}$ be the image of $F_0$ by $B_{k-1}\circ\ldots\circ  B_0$.
In particular, the projective map $\PP(B_{n-1}\ldots B_0)$ sends isometrically the circle $\PP E_0$ on the circle $\PP E_{n-1}$.

Then each map $B_k$ induces a quotient isomorphism $\tilde{B}_k$ between the Euclidean quotient spaces $V_{T^kx}/F_k$ and $V_{T^{k+1}x}/F_{k+1}$. Since $F_0\perp E_0$ and  $F_{n-1}\perp E_{n-1}$, the composition 
$\tilde{B}_{n-1}\circ\ldots\circ  \tilde{B}_0$ identifies then to the restriction $B^n(x)_{|E_0}$. 
Applying again 
 \cref{c.singvalues} one obtains a sequence $\tilde{C}_k$ of $\epsilon$-perturbations of  the linear maps $\tilde{B}_k$ such that the ratio between the two singular values of the product  $\tilde{C}_{n-1}\circ\ldots\circ  \tilde{C}_0$ is of the order of $e^{n\epsilon}$. The linear maps $\tilde{C}_k$ may then be lifted into perturbations $C_k$ of the linear maps $B_k$, such that $C_k$ sends $F_k$ on $F_{k+1}$. 
 Using partitions of unity, we realize the sequences $B_k$ and $C_k$ of perturbations of the sequence of linear maps $A_k$ by two continuous linear cocycles $\cB$ and $\cC$ that are $\epsilon$-perturbations of $\cA$, so that 
 \begin{itemize} 
 \item $\PP\cB^n$ sends the circle $\PP E_0$ isometrically on the circle $\PP E_{n-1}$,
 \item $\cC^n$ sends $F_0$ on $F_{n}$ and the induced quotient map from $V_{x}/F_0$ to $V_{T^nx}/F_n$ has nonconformity on the order of $e^{n\epsilon}$.
 \end{itemize}
Then if $N$ is large (recall that we have $n\geq N$) and if there is a homeomorphism $h_k$ of $\PP V$ fibering on $\Id_X$ and conjugating $\PP\cB$ to $\PP\cC$, that homeomorphism cannot be close to identity both by restriction to the circle $\PP E_0$ and to the circle $\PP E_{n-1}$. 
Thus $\PP\cA$ is not  structurally stable.
\end{proof}

\appendix

\section{Comparison with  Avila and Krikorian's 'variation of the fibered rotation numbers'  for  $\SL(2,\RR)$-cocycles.}\label{a.avilakrikorian}
For the sake of completeness, we reproduce here (with some notational simplifications) the construction by Avila-Krikorian~\cite{AK} for  $\SL(2,\RR)$-cocycles of a 'variation of the fibered rotation numbers'. We show that it coincides with our rotation number, up to multiplication by $-2$.

\bigskip

\noindent {\em Setting:} consider a path of $\SL(2,\RR)$-valued linear cocycles $\tA=\bigl(\cA_t\bigr)_{t\in [0,1]}$ over a fixed dynamics on the base $T\colon X\to X$. Through complexification, Avila and Krikorian define a $\CC$-valued additive cocycle over some continuous fibered extension 
$$T_{\tA}\colon X\times \DD\times \DD\to X\times \DD\times \DD$$ of $T$ whose Birkhoff averages tend a.e. to a function $\zeta_{\tA}\colon X\times \DD\times \DD\to \CC$. They show that this function only depends on $x\in X$. Moreover, wherever defined,
\begin{enumerate}
\item its real part is, up to multiplication by $-2$, the average growth rate of the winding number of the paths of matrices $\bigl(\cA^k_t(x)\bigr)_t$, that is, the rotation number of the path $\tA$ at $x$, as we defined it in this paper,\label{i.windingAK1}
\item  its imaginary part is the Lyapunov exponent at $x$.\label{i.LyapAK2}
\end{enumerate}

\medskip

We reproduce here their proof of convergence of the real part and we explain \cref{i.windingAK1}. 
\bigskip

\subsection{Avila-Krikorian fibered rotation number for $\SU(1,1)$-cocyles.}
Consider the natural projective action of $\SL(2,\CC)$ on the Riemann sphere $\overline{\CC}=\PP(\CC^2)$ by M\"obius transformations. The subgroup of matrices in $\SL(2,\CC)$ that leave invariant the Poincar\'e disk $\DD=\{[z:1], |z|<1\}$ is the group $\SU(1,1)$ of matrices that preserve the quadratic form $q(x,y)=|x|^2-|y|^2$ on $\CC^2$.  

Denote by $(A,z)\mapsto A\cdot z$ the corresponding action on $\overline{\DD}$ identified to the closed unit disk of $\CC$.
Given a matrix $A\in \SU(1,1)$ and $z\in \overline{\DD}$, one has in $\CC^2$ the equality $A \left(\begin{array}{c}z\\
1
\end{array}\right)=\tau_{A}(z)\cdot \left(\begin{array}{c}A\cdot z\\
1
\end{array}\right)$
for some unique complex number $\tau_{A}(z)\in \CC\setminus\{0\}$. This satisfies a multiplicative cocycle relation:
\begin{equation}
\tau_{AB}(z)=\tau_{A}(B\cdot z)\tau_{B}(z)\label{e.multcocycle}
\end{equation}

A path $\bigl(A_t\bigr)_{t\in[0,1]}$ of matrices in $\SU(1,1)$ and a path $\bigl(z_t\bigr)_{t\in[0,1]}$ in $\overline{\DD}$ give then a path  $\tau_{A_t}(z_t)$ in $\CC\setminus\{0\}$. Take any lift $\gamma_t$ of that path by the map $z\mapsto e^{i2\pi z}$, then the difference $\delta_{z_0,z_1}=\gamma_1-\gamma_0\in \CC$ only depends on the isotopy class with fixed extremities of $\bigl( A_t\bigr)$ and, by simple-connectedness of $\overline{\DD}$,  on $z_0$ and $z_1$. Note that the real part $\Re\delta_{z_0,z_1}$  is the winding number of the path $\tau_{A_tx}(z_t)$ around $0\in \CC$.
Write
$A_t=\left(\begin{array}{cc}u_t&\bar{v}_t\\
v_t&\bar{u}_t
\end{array}
\right), u_t,v_t\in \CC$, then $\tau_{A_t}(z)=v_tz+\bar{u}_t$. In particular, the connected set $\tau_{A_t}( \overline{\DD})$ does not meet the line $\RR i\bar{u}_t$. Thus, given two paths $z_t$ and $z'_t$ in $\overline{\DD}$, the difference between the winding numbers of $\tau_{A_tx}(z_t)$ and $\tau_{A_tx}(z'_t)$ is less than $1$, that is, 
\begin{equation}\label{e.deltaofwinding}
\left|\Re\left(\delta_{z_0,z_1}-\delta_{z'_0,z'_1}\right)\right|<1.
\end{equation}

Consider now a path of $\SU(1,1)$-cocycles $\tA=\bigl(\cA_t=(T,A_t)\bigr)_{t\in [0,1]}$, where $T\colon X\to X$ is a continuous map and  $A_t\colon X \to \SU(1,1)$ is a path of continuous maps.
Define the continuous map $$T_\tA\colon X\times \overline{\DD}\times  \overline{\DD} \to X\times  \overline{\DD}\times  \overline{\DD}$$ by $T_\tA(x,z_0,z_1)=(Tx,A_0x\cdot z_0, A_1x\cdot  z_1)$. Define an additive cocycle 
$(T_\tA,\delta_\tA)$: $$\delta_\tA\colon X\times \overline{\DD}\times  \overline{\DD}\to \CC$$  by taking for $\delta_\tA(x,z_0,z_1)$ the complex number $\delta_{z_0,z_1}$ obtained as previously with the path of matrices  
$\bigl(A_tx\bigr)$ and any path $z_t$ joining $z_0$ to $z_1$ in $\overline{\DD}$. This is a well-defined cocycle.

As a consequence of \cref{e.multcocycle}, $(T_\tA,\delta_\tA)^n$ is the cocycle $(T_{\tA^n},\delta_{\tA^n})$, where $\tA^n$ is the path of iterated cocycles $\bigl(\cA^n_t\bigr)$. With \cref{e.deltaofwinding}, this implies that the real part of the $n$-th  Birkhoff average of the cocycle $(T_\tA,\delta_\tA)$ satisfies 
$$|B_n(x,z_0,z_1)-B_n(x,z'_0,z'_1)|<1/n,$$
for any numbers $z_0,z_1,z'_0,z'_1\in \overline{\DD}$.
It follows from the Birkhoff Ergodic Theorem that, for any $T$-invariant measure $\mu$ supported in $X$, the real part of the Birkhoff averages converge $\mu$-a.e. to a function that depends only on $x\in X$, which we may denote by $\delta\rho_\tA\colon X\to \RR$. 

\subsection{The case of $\SL(2,\RR)$-cocycles, and correspondance with our rotation number. }
Recall that $\SL(2,\RR)$ is the set of matrices $A$ whose projective action fixes the upper half-plane $$\HH=\{[z:1], \Im z> 0\}\subset \PP \CC^2.$$
Multiplication by the matrix 
$$Q=\frac{-1}{1+i}\left(\begin{array}{cc}
1& -i\\
1& i
\end{array}\right)\in \SL(2,\CC)$$
sends $\HH$ on the Poincar\'e disk and conjugacy $A\mapsto QAQ^{-1}$ is an isomorphism from $\SL(2,\RR)$ to $\SU(1,1)$.

With a path of $\SL(2,\RR)$-cocycles $\tA=\bigl(\cA_t=(T,A_t)\bigr)_{t\in [0,1]}$, associate the path   of $\SU(1,1)$-cocycles $\mathring{\tA}=\bigl(T,QA_tQ^{-1}\bigr)_{t\in [0,1]}$. The authors (in slightly different notations) call $\delta\rho_{\mathring{\tA}}$ the {\em average variation of the fibered rotation number} of the path of linear cocycles $\tA$. 

Multiplication by $Q$ sends the boundary $\partial \HH$ on the boundary $\partial \DD$ by an orientation preserving isometry. That is, given $w$ in the circle $\PP(\RR^2)\equiv \partial \HH$, the winding number of the path $\PP\cA_t \cdot w$ is the winding number of the path $\mathring{\tA}\cdot Qw$ in  $\partial \DD$.
Given a path $A_t=\left(\begin{array}{cc}u_t&\bar{v}_t\\
v_t&\bar{u}_t
\end{array}
\right)$ in $\SU(1,1)$, we have $A_t\cdot 1=\frac{u_t+\bar{v_t}}{v_t+\bar{u_t}}$, that is, the winding number of the path $A_t\cdot 1$ of $\partial \DD$ is $-2$ times the winding number of the path $\tau_{A_t}(1)=v_t+\bar{u_t}$ in $\CC$.

Thus, given a path of $\SL(2,\RR)$-cocycles $\tA=\cA_t$, we have the following relation between our rotation number and that of \cite{AK}:
$$\rho(\cA_t,x)=-2\delta\rho_{\mathring{\tA}},$$
which is what we wanted to prove.

\end{document}